\documentclass[11pt,reqno]{amsart}

\usepackage{amsmath,amssymb,mathrsfs}
\usepackage{graphicx,cite,times}
\usepackage{cases}
 \usepackage{dsfont}
 \usepackage{color}
\newtheorem{theo}{Theorem}[section]

\newtheorem{lemm}[theo]{Lemma}
\newtheorem{prop}[theo]{Proposition}
\newtheorem{rema}[theo]{Remark}

\numberwithin{equation}{section}

\begin{document}

\title[Quasi-periodic travelling waves for damped  beams]{
Quasi-periodic travelling waves for a class of damped  beams  on rectangular tori}

\author{Bochao Chen}
\address{College of Mathematics, Jilin University, Changchun, Jilin 130012, P.R.China}
\email{chenbc758@163.com}

\author{Yixian Gao}
\address{School of
Mathematics and Statistics, Center for Mathematics and
Interdisciplinary Sciences, Northeast Normal University, Changchun, Jilin 130024, P.R.China}
\email{gaoyx643@nenu.edu.cn}

\author{Juan J. Nieto}
\address{Departamento de Estadistica, Analisis Matematico y Optimizacion and instituto de matematicas, Universidade de Santiago de
Compostela, 15782 Santiago, Spain}
\email{juanjose.nieto.roig@usc.es}

\thanks{The research of BC was partially supported by  NSFC grant 11901232 and China Postdoctoral Science Foundation Funded Project
2019M651191. The research of YG was partially supported by NSFC grant 11871140,
JJKH20180006KJ and FRFCU2412019BJ005. The research of JJN was partially  supported by the Agencia Estatal de Investigacion (AEI) of Spain, co-financed by the European Fund for Regional Development (FEDER) corresponding to the 2014¨C2020 multiyear financial framework, project MTM2016-75140-P . Moreover, Nieto also thanks partial financial support by Xunta de Galicia under grant ED431C 2019-02.}



\begin{abstract}
This article concerns a class of beam equations with damping on rectangular tori. When the generators satisfy  certain relationship, by excluding some value of two model parameters, we prove that such models admit small amplitude quasi-periodic travelling wave solutions with two frequencies, which are continuations of two rotating wave solutions  with one frequency.  This result holds  not only for an isotropic torus, but also for an anisotropic torus. The proof  is mainly based on a  Lyapunov--Schmidt reduction together with the implicit function theorem.
\end{abstract}

\maketitle

\section{Introduction}

Consider nonlinear beam equations with damping on rectangular tori
\begin{align}\label{a1}
u_{tt}+\mu\Delta^2u+\alpha u_t+\gamma\Delta^2 u_t+mu=\lambda(u_t)^{2p+1},\quad t\in\mathbb{R},x\in\mathbb{T}^2_{\mathrm{L}},
\end{align}
where $\Delta^2$ is the biharmonic operator with $\Delta=\sum^2_{k=1}\partial^2_{x_k}$, and $p$ is a positive integer.
The model parameter $\mu:= EI/\rho>0$ denotes the elasticity
coefficient, where $\rho$ is the  mass density,   $I$ is the second moment of
area of the beam's cross section,  $E$ is Young's modulus of elasticity and  the product $EI$ is the flexural rigidity. The other
$m >0$ describes the linear stiffness
of the foundation, and the real numbers $\alpha,\gamma,\lambda$ are coefficients of friction. In our model, $\lambda$ is taken as a positive number. The restriction on $\alpha,\gamma$ will be  given later.

In recent years many efforts have been made to understand some properties of beam equations  on the standard
tori $\mathbb T^d = \mathbb R^d/(2\pi \mathbb Z)^d,d\geq1.$  In this paper we intend to give some results on rectangular  tori.
Denote by $\mathbb{T}^2_{\mathrm{L}}$ a 2-dimensional rectangular torus as follows
\begin{align*}
\mathbb{T}^2_{\mathrm{L}}:=(\mathbb{R}/2\pi \mathrm{L}_1\mathbb{Z})\times(\mathbb{R}/2\pi \mathrm{L}_2\mathbb{Z})
\end{align*}
with generators $\mathrm{L}_1>0$, and $\mathrm{L}_2>0$. If  $\mathrm{L}_1$ is ``rationally independent'' of $\mathrm{L}_2$, then $\mathrm{L}_1$ cannot be written as $\mathrm{L}_1=a\mathrm{L}_2$, with a rational coefficient $a$. In this case the corresponding torus is called anisotropic. On the other hand, when $\mathrm{L}_1,\mathrm{L}_2$ are  rational,  the rectangular torus $\mathbb{T}^2_{\mathrm{L}}$ can be reduced to the standard torus $\mathbb{T}^2$ by a simple geometric consideration. That is, by setting $\mathrm{L}_k=a_k/b_k$ for some $a_k,b_k\in\mathbb{N},k=1,2$, then we obtain the scaled standard torus $a\mathbb{T}^2=\mathbb{R}^2/(2\pi a \mathbb{Z})^2 $, where $a$ is taken as the least  common multiple of $a_k,k=1,2$. We are interest in the existence of quasi-periodic traveling wave solutions of the form
\begin{align}\label{c1}
u(t,x)=\varphi(\omega_1t+\mathrm{L}^{-1}_1x_1,\omega_2t+\mathrm{L}^{-1}_2x_2),
\end{align}
where $\omega=(\omega_1,\omega_2)\in\mathbb{R}^2$ are ``rationally independent'' and $\varphi$ is a $(2\pi)^2$-periodic function.

\subsection{Backgrounds and Main Ideas}

The classical linear theory of deformation derives the following Euler--Bernoulli equation
\begin{align*}
\rho(x)u_{tt}+(EI(x)u_{xx})_{xx}=0, \quad t\in\mathbb{R},x\in \Omega\subset\mathbb{R},
\end{align*}
which describes the motion of thin elastic beams for at least two hundred years old, see \cite{weaver1990vibration}. From the physical point of view, in a real process, dissipation plays an important spreading role for
the energy gather arising from the nonlinearity, and the interaction of it with the nonlinearity accompanies the accumulation, balance and dissipation of the energy in the configurations, and governs the
longtime behavior of the dynamical system associated with the corresponding nonlinear  equation. There are two types of distinguished mechanisms. One is external such as ``interaction with
surrounding medium, interface with other physical systems''(\cite{herrmann2008vibration}). The other is internal such as ``increase of heat energy to the detriment of mechanical
energy by virtue of internal friction, thermoelastic
effects''(\cite{herrmann2008vibration}).  We call dissipation mechanisms direct if supplementary dissipation
terms are directly inserted into the original conservative equations. On the
other hand, dissipation mechanisms are called indirect for ``coupling
the mechanical equations governing beam motion to
related dissipation systems with additional dynamics, resulting
in an overall system in which mechanical energy
is dissipated''(\cite[p.379]{Russell1991comparison}). We refer the readers to \cite{Russell1991comparison} for such coupled dissipative systems such as  Euler--Bernoulli  beams with thermoelastic damping and
with shear diffusion damping. Two kinds of linear dissipation mechanisms contained in our model  are direct, see \cite{herrmann2008vibration}. We call the term $u_t$  external or viscous
damping introduced by external, linear dampers. The term $\Delta^2 u_t$ is called internal
or Kelvin--Voigt damping which describes a situation where higher frequencies are more strongly damped than low ones. Moreover, this model involves a nonlinear term depending on velocity. It is worth mentioning that Kosovali\'{c} \cite{kosovali2018quaisperiodic} have investigated quasi-periodic travelling waves for  equation \eqref{a1} with $\lambda=1$ on a 2-dimensional standard torus and  established the existence of small amplitude quasi-periodic travelling wave solutions with two frequencies, which are continuations of rotating wave solutions. Some recent results on damped beam equations can be found in \cite{wang1991asymptotic,liu1998exponential,denk2015structural} and the references there in.

Concerning the rectangular tori or irrational tori, these papers can be mainly divided into two groups. In the first group,
one can consult the articles \cite{bourgain2007strichartz,bourgain2015proof,deng2017strichartz} for Strichartz estimates
and well-posedness of PDEs.
In \cite{deng2017strichartz}, Deng et al. got the Stricharz estimates over large time scale for
Schr\"{o}dinger equations on \emph{generic} rectangular tori.
The other concerns phenomenon of growth of Sobolev norms of PDEs,
for instance \cite{deng2019growth,berti2019long,deng2019growth2}.
In addition, concerning the periodic/quasi-periodic problem,
 Rabinowitz \cite{Rabinowitz1967periodic,Rabinawitz1968periodic} established
 for the first time the existence of  periodic solutions of  forced dissipative wave equations with respect to Dirichlet boundary conditions
\begin{align*}
&u_{tt}-u_{xx}+\alpha u_t=\epsilon f(t,x,u,u_t,u_x),\quad\alpha\neq0,\\
&u_{tt}-u_{xx}+\alpha u_t=\epsilon f(t,x,u,u_t,u_x,u_{tt},u_{tx},u_{xx}),\quad \alpha\neq0.
\end{align*}
Later on, Bourgain \cite{bourgain1999periodic}  provided a detailed proof on the existence of periodic solutions for the following wave equation
\begin{align*}
u_{tt}-u_{xx}+mu+u^2_t=0,\quad m\neq0.
\end{align*}
 In \cite{Calleja2017response}, Calleja et al. obtained response solutions
(i.e., quasi-periodic solutions with the same frequency as the forcing) of several models of nonlinear wave equations with very strong damping, such as $\frac{1}{\epsilon}u_t$ for $\epsilon$ small enough. Recently,
Saha et al. \cite{saha2020can} treated the influence of quasi-periodic gravity waves in the shape of crystals in clouds.
Berti--Montalto \cite{berti2020quasi} proved the existence of small amplitude quasi-periodic standing wave solutions of an ocean with infinite depth under the action of gravity and surface tension.
In \cite{chen2019periodic}, Chen--Pelinovsky presented a conjecture on the possible generalization to the case of quasi-periodic solutions to a  general periodic travelling wave of the modified Korteweg-de Vries equation.
 To the best of our knowledge, there is few article on the existence
of quasi-periodic travelling wave solutions for beam equations with damping on rectangular tori so far.

We define the energy functional $\mathcal{E}$ associated with equation \eqref{a1} by
\begin{align*}
\mathcal E(t)=\int_{\mathbb{T}^2_{\mathrm{L}}}\frac{1}{2}(\partial_tu(t,x))^2+\frac{1}{2}\mu(\Delta u(t,x))^2+\frac{1}{2}m(u(t,x))^2\mathrm{d}x.
\end{align*}
Using periodic boundary conditions and integration by parts yields that
\begin{align*}
\frac{\mathrm{d}}{\mathrm{d}t}\mathcal E(t)=&\int_{\mathbb{T}^2_{\mathrm{L}}}\partial_tu(t,x)(-\mu\Delta^2u(t,x)-\alpha\partial_tu(t,x)-\gamma\partial_t\Delta^2 u(t,x)- mu(t,x))\\
&\quad\quad+\lambda(\partial_tu(t,x))^{2p+2}+\mu\Delta u(t,x)\partial_t\Delta u(t,x)+mu(t,x)\partial_tu(t,x)\mathrm{d}x\\
=&\int_{\mathbb{T}^2_{\mathrm{L}}}-\alpha(\partial_tu(t,x))^2-\gamma\partial_tu(t,x)\partial_t\Delta^2 u(t,x)+\lambda(\partial_tu(t,x))^{2p+2}\mathrm{d}x\\
=&\int_{\mathbb{T}^2_{\mathrm{L}}}-\alpha(\partial_tu(t,x))^2-\gamma\partial_tu(t,x)\Delta^2 \partial_t u(t,x)+\lambda(\partial_tu(t,x))^{2p+2}\mathrm{d}x.
\end{align*}
Since $\Delta^2$ is positive semi-definite and $\lambda$ is positive, if $\alpha\leq0,\gamma\leq0$, then $\mathcal E$ is increasing on $t$. This means that there may be a nontrivial smooth periodic/quasi-periodic solution to equation \eqref{a1} for  either $\alpha>0$, or $\gamma>0$. In our study, we regard the friction coefficients $\alpha,\gamma$ as the bifurcation parameters and take advantage of their smallness. When  generators $\mathrm{L}_1$ and $\mathrm{L}_2$ satisfy  certain relationship, by excluding some value of two model parameters, there exists a sequence of Hopf bifurcations of quasi-periodic solutions for the beam model \eqref{a1}.
More precisely, for some fixed positive generators $\mathrm{L}_1,\mathrm{L}_2$, when $(\mu,m)$ are located  in a dense subset of $\mathbb{R}^{+}\times\mathbb{R}^{+}$, we construct quasi-periodic travelling wave solutions with two frequencies, which are continuations of rotating wave solutions.  Because of Hopf bifurcations, model \eqref{a1} actually describes the motion of a $2$-dimensional beam  subject to self-oscillation, which is also known as ``maintained'', ``sustained'', ``self-excited'', ``self-induced'', ``spontaneous'', ``autonomous'', and ``hunting'' or ``parasitic'' vibration, see \cite{jenkins2013self}.  A self-oscillator can generate and maintain a regular mechanical
periodicity or quasi-periodicity without requiring a similar external periodicity or quasi-periodicity to drive it.
For example,  time delay can cause periodic and quasi-periodic vibrations of vehicle wheels in the experiments of  \cite{takacs2009experiments}.
In mechanical engineering,  vibrations of a mechanical system can be induced by an external source
which acts on the system autonomously.   The phenomenon of self-oscillation is prevalent such as the heartbeat,
the pupil light reflex,  clocks,  heat engines and so on, see \cite{jenkins2013self,longtin1989modelling}.
In \cite{Campbell1995complex,campbell1995limit}, Campbell et al. showed that a harmonic oscillator subject to
damping and time delay  undergoes the bifurcation of not only periodic solutions, but also $2$-dimensional tori.
There are a sequence of Hopf bifurcations for  wave equations with damping and without  external forcing on $1$-dimensional or $d$-dimensional tori, see \cite{kosovalic2019self,kosovalic2019excited}.

Compared with the results on  standard tori given by \cite{kosovali2018quaisperiodic},
the analysis of damped beam equations on  rectangular tori presents significant new obstacles.
We still apply common branching methods including  both a Lyapunov--Schmidt reduction and the implicit function theorem to study the existence of quasi-periodic solutions.
We first linearize  model \eqref{a1} around zero for the friction coefficients $\alpha,\gamma$ vanish.
Then we can obtain an equivalent equation associated with Fourier coefficients.
We further consider the equivalent equation by fixing two frequencies which are rationally independent.
Because two extra parameters corresponding to the rectangular torus are introduced,
this increases the difficulty of solving  the equivalent equation with fixed frequencies. For this, we restrict generators $\mathrm{L}_1$ and $\mathrm{L}_2$ to certain relationship.
Moreover, we have to impose some conditions on model parameters $(\mu,m)$. In particular, for $\mathrm{L}_1=\mathrm{L}_2=1$, the set of model parameters $(\mu,m)$ can be written as
\begin{align*}
\textstyle\bigg\{(\mu,m)\in\mathbb{R}^+\times\mathbb{R}^{+}:
\mu^2\in\mathbb{Q}^+,m^2\in\mathbb{Q}^+,\frac{m}{\mu}\notin\mathbb{N},
\left(\frac{\mu(j^*_1)^4+m}{\mu (j^*_{2})^4+m}\right)^{\frac12}
\in\mathbb{R}^+\backslash\mathbb{Q}\bigg\}.
\end{align*}
That is, by excluding fewer value of two model parameters than in \cite{kosovali2018quaisperiodic}, we can construct small amplitude quasi-periodic travelling wave solutions with two frequencies for model \eqref{a1} on a standard torus.
 Our next purpose is to apply the Lyapunov--Schmidt reduction to obtain a bifurcation equation and a range equation. Finally,  we can solve the  bifurcation equation and the range equation by the implicit function theorem, respectively. In our analysis, we are able to avoid  small divisor problems at infinity owing to the biharmonic operator.

\subsection{Main results}

In this subsection we will introduce two main theorems. The first theorem corresponds to the existence of quasi-periodic travelling wave solutions. The other  addresses the ``directions of bifurcation''.

For fixed positive generators $\mathrm{L}_1,\mathrm{L}_2$, and fixed positive integers $j_1^*\neq \frac{\mathrm{L}_1}{\mathrm{L}_2}j^*_{2}$, the sets of model parameters $(\mu,m)$ are defined  as follows
\begin{align*}
\textstyle\mathcal{S}_{\mu,m}:=\bigg\{(\mu,m)\in\mathbb{R}^{+}\times\mathbb{Q}^{+}:\mu \mathrm{L}^{-4}_1\in\mathbb{Q}^{+},
\frac{m}{\mu \mathrm{L}^{-4}_1}\notin\mathbb{N},\frac{m}{\mu{\mathrm{L}}^{-4}_2}\notin\mathbb{N},
\textstyle\left(\frac{\mu \mathrm{L}^{-4}_1(j^*_1)^4+m}{\mu {\mathrm{L}}^{-4}_2(j^*_{2})^4+m}\right)^{\frac12}
\in\mathbb{R}^+\backslash\mathbb{Q},\nonumber\\
\textstyle\frac{\mathrm{L}^2_2}{\mathrm{L}^2_1}\left(\frac{\mu {\mathrm{L}}^{-4}_1(j^*_1)^4+m}{\mu {\mathrm{L}}^{-4}_2(j^*_{2})^4+m}\right)^{\frac12}
\in\mathbb{R}^+\backslash\mathbb{Q}, \text{ for } \mathrm{L}_1\in\mathbb{R}^+,\mathrm{L}_2\in\mathbb{R}^+,\text{ with }{\mathrm{L}^4_1}/{\mathrm{L}^4_2}\in\mathbb{Q}^+\bigg\}
\end{align*}
and
\begin{align*}
\textstyle\mathcal{S}'_{\mu,m}:=\bigg\{(\mu,m)\in\mathbb{R}^+\times\mathbb{R}^{+}:
\mu^2\in\mathbb{Q}^+,m^2\in\mathbb{Q}^+,\frac{m}{\mu}\in\mathbb{R}^{+}\backslash\mathbb{Q}, \text{ for } \mathrm{L}_1\in\mathbb{R}^+,\mathrm{L}_2\in\mathbb{R}^+,&\\
\text{with } \mathrm{L}^{4}_1\in\mathbb{Q}^+,
\mathrm{L}^{4}_2\in\mathbb{Q}^+\text{ and } {\mathrm{L}^{2}_1}/{\mathrm{L}^{2}_2}\in\mathbb{Q}^+\bigg\}&,
\end{align*}
where $\mathbb{R}^{+}:=\{x\in\mathbb{R}:x>0\},\mathbb{Q}^{+}:=\{x\in\mathbb{Q}:x>0\}$.
\begin{theo}\label{theo1}
Let $\lambda>0$. Fix positive generators $\mathrm{L}_1,\mathrm{L}_2$, and positive integers $j_1^*\neq \frac{\mathrm{L}_1}{\mathrm{L}_2}j^*_{2}$.  If $(\mu,m)$ belong to the set $\mathcal{S}_{\mu,m} $ ({\rm resp.} $\mathcal{S}'_{\mu,m}$), then for amplitudes $\rho=(\rho_1,\rho_2)\simeq0$, for $(\alpha,\gamma)=(\alpha(\rho),\gamma(\rho))\approx0$, with frequencies $\omega=(\omega_1,\omega_2)$ near $\omega_{j^*}=(\omega_{j_1^*},\omega_{j_2^*})$, where
\begin{align*}
\textstyle\omega_{j^*_1}=\left(\frac{\mu \mathrm{L}^{-4}_1(j^*_1)^4+m}{(j^*_1)^2}\right)^{\frac12},\quad\textstyle\omega_{j^*_2}=\left(\frac{\mu \mathrm{L}^{-4}_2(j^*_2)^4+m}{(j^*_2)^2}\right)^{\frac12},
\end{align*}
equation \eqref{a1} admits a  family of small amplitude quasi-periodic travelling wave solutions with two frequencies of the form
\begin{align*}
u(\rho)(t,x)=\varphi(\rho)(\omega_1t+\mathrm{L}^{-1}_1x_1,\omega_2t+\mathrm{L}^{-1}_2x_2),
\end{align*}
where $\varphi(\rho)$ is a real valued function with $(2\pi)^2$-period of the form
\begin{align*}
\textstyle\varphi(\rho)(\omega_1t+\mathrm{L}^{-1}_1x_1,\omega_2t+\mathrm{L}^{-1}_2x_2)=&2\rho_1\cos (\omega_1t+\mathrm{L}^{-1}_1x_1)+2\rho_2\cos(\omega_2t+\mathrm{L}^{-1}_2x_2)\\
&+w(\rho)(\omega_1t+\mathrm{L}^{-1}_1x_1,\omega_2t+\mathrm{L}^{-1}_2x_2),
\end{align*}
where $w(\rho)\in C^{\infty}(\mathbb{T}^2;\mathbb{R})$. The mapping $\rho\longmapsto \varphi(\rho)\in H^s$ is $C^\infty$ for all $s>0$.

Moreover, these quasi-periodic travelling wave solutions of equation \eqref{a1} branch off of rotating wave solutions, in the sense that setting one of the amplitudes to zero, gives a family of rotating wave solution of equation \eqref{a1}. More precisely,
\begin{align*}
&u(0,\rho_2)(t,x)=\varphi(0,\rho_2)(\omega_2t+\mathrm{L}^{-1}_2x_2),\\
&u(\rho_1,0)(t,x)=\varphi(\rho_1,0)(\omega_1t+\mathrm{L}^{-1}_1x_1),
\end{align*}
which are $2\pi$-periodic, respectively.

Finally, the set $\mathcal{S}_{\mu,m}$ ({\rm resp.} $\mathcal{S}'_{\mu,m}$) of parameters $(\mu,m)$ is dense in the  space $\mathbb{R}^{+}\times\mathbb{R}^{+}$.
\end{theo}
\begin{rema}
Observe that $\mathcal{S}_{\mu,m}\cap\mathcal{S}'_{\mu,m}=\emptyset$. In fact, since $j_1^*\neq \frac{\mathrm{L}_1}{\mathrm{L}_2}j^*_{2}$, one has
\begin{align*}
&\textstyle\frac{\mu \mathrm{L}^{-4}_1(j^*_1)^4+m}{\mu {\mathrm{L}}^{-4}_2(j^*_{2})^4+m}\in\mathbb{Q}^{+}\qquad \text{ if } (\mu,m)\in\mathcal{S}_{\mu,m},\\
&\textstyle\frac{\mu \mathrm{L}^{-4}_1(j^*_1)^4+m}{\mu \mathrm{L}^{-4}_2(j^*_{2})^4+m}=\frac{\mu^2\mathrm{L}^{-4}_1\mathrm{L}^{-4}_2(j^*_1)^4(j^*_2)^4-m^2
+m\mu(\mathrm{L}^{-2}_2(j^*_2)^2+\mathrm{L}^{-2}_1(j^*_1)^2)(\mathrm{L}^{-1}_2j^*_2+\mathrm{L}^{-1}_1j^*_1)
(\mathrm{L}^{-1}_2j^*_2-\mathrm{L}^{-1}_1j^*_1)}{\mu^2\mathrm{L}^{-8}_2(j^*_2)^8-m^2}
\in\mathbb{R}^+\backslash\mathbb{Q}\\
&\qquad\qquad\qquad\qquad\text{ if } (\mu,m)\in\mathcal{S}'_{\mu,m}.
\end{align*}
\end{rema}

\begin{theo}\label{theo2}
Under the same conditions in Theorem \ref{theo1},  if there exist nonconstant quasi-periodic travelling wave solutions to equation \eqref{a1}, then either $\alpha>0$, or $\gamma>0$.
\end{theo}
\begin{proof}
The theorem follows from an energy argument as above.
\end{proof}

The outline of this paper is as follows. In Section \ref{sec:2}, we obtain a bifurcation equation together with a range equation  by virtue of the Lyapunov--Schmidt reduction. The projects of Sections \ref{sec:3}--\ref{sec:4} are to solve the range equation and the bifurcation equation via the classical implicit function theorem, respectively. Moreover, we show the  $C^{\infty}$ smoothness of the solutions obtained with respect to the amplitude parameters.

\section{Lyapunov--Schmidt reduction}\label{sec:2}

If we set  $\theta=(\theta_1,\theta_2)$ with $\theta_k=\omega_k t+\nu_kx_k,\nu_k=\mathrm{L}^{-1}_k,k=1,2$, then substituting  the ansatz \eqref{c1} into equation \eqref{a1} yields that
\begin{align}\label{c2}
(\omega\cdot\nabla)^2\varphi+\mu\Delta^2_{\nu}\varphi+\alpha(\omega\cdot\nabla)\varphi+
\gamma(\omega\cdot\nabla)\Delta^2_{\nu}\varphi+m\varphi=\lambda((\omega\cdot\nabla)\varphi)^{2p+1}&,
\end{align}
where
\begin{align*}
&\Delta_{ \nu}=\sum^2_{k=1}\nu^2_{k}\partial^2_{\theta_k},\quad \nu:=(\nu_1,\nu_2),\text{ with }\nu_{k}=\mathrm{L}^{-1}_{k},k=1,2.
\end{align*}
Clearly, the eigenvalues of $-\Delta_{\nu}$ are
\begin{align*}
 \lambda_{j}=\sum_{k=1}^2\nu^2_kj^2_k,\quad j=(j_1,j_2)\in\mathbb{Z}^2.
\end{align*}
The main focus of the present subsection is to decompose the equivalent equation  \eqref{c2} as  a bifurcation equation and a range equation by the Lyapunov--Schmidt reduction.

For $s\geq0$, denote by $H^s$ the Sobolev space  of real-valued functions
\begin{align*}
H^s:=H^s(\mathbb{T}^2;\mathbb{R})=\left\{\varphi(\theta)=\sum\limits_{j\in\mathbb{Z}^2}\varphi_{j}e^{\mathrm{i}j\cdot \theta}: \overline{\varphi_{j}}=\varphi_{-j},\|\varphi\|^2_{s}:=\sum\limits_{j\in\mathbb{Z}^2}(1+|j|^{2s})|\varphi_{j}|^2<\infty\right\},
\end{align*}
where $\overline{\varphi_{j}}$ is the complex conjugate of $\varphi_{j}$ and
\begin{align*}
|j|^2=j\cdot j=\sum_{k=1}^2j_k^{2},\quad \varphi_{j}=\frac{1}{4\pi^2}\int_{\mathbb{T}^2} \varphi(\theta)e^{-\mathrm{i}j\cdot\theta}\mathrm{d}\theta.
\end{align*}
Clearly, the space $(H^s,\|\cdot\|_s)$ is a Banach space as well as a Hilbert space.
If $s>1$, then the space $H^s$ is a Banach algebra with respect to multiplication of functions, that is,
\begin{align*}
\|\varphi_1\varphi_2\|_{{s}}\leq C(s)\|\varphi_1\|_{s}\|\varphi_2\|_{s}, \quad\forall \varphi_1,\varphi_2\in{H}^s.
\end{align*}
For $s>k+1$ with $k\in\mathbb{N}$, one  has
\begin{align*}
H^s\subset C^{k}(\mathbb{T}^2;\mathbb{R})\quad\text{and}\quad\|\varphi\|_{C^k}\leq C\|\varphi\|_s,
\end{align*}
where $C^k(\mathbb{T}^2;\mathbb{R})$ denotes the set consisted  of $k$ times differentiable functions on $\mathbb{T}^2$, with values in $\mathbb{R}$.

Define the linear differential operator
\begin{align*}
L_{\omega,\alpha,\gamma}:H^{s+5}&\longrightarrow H^{s},\\
\varphi&\longmapsto(\omega\cdot\nabla)^2\varphi+\mu\Delta^2_{\nu}\varphi+\alpha(\omega\cdot\nabla)\varphi+
\gamma(\omega\cdot\nabla)\Delta^2_{\nu}\varphi+m\varphi,
\end{align*}
 and the composition operator
\begin{align*}
&F:(\omega,\varphi)\longmapsto \lambda((\omega\cdot\nabla)\varphi)^{2p+1}.
\end{align*}
Then we can rewrite  \eqref{c2}  as
\begin{align}\label{f:eq}
L_{\omega,\alpha,\gamma}\varphi=F(\omega,\varphi).
\end{align}
In order to solve \eqref{f:eq}, we have to introduce the smoothness of the above composition operator.
\begin{lemm}\label{le:smoothness}
Let $p$  be a positive integer. For $\omega=(\omega_1,\omega_2)\in\mathbb{R}^2$, define
\begin{align*}
F:\mathbb{R}^2\times H^s\longrightarrow H^{s-1},\quad (\omega,\varphi)\longmapsto \lambda((\omega\cdot\nabla)\varphi)^{2p+1}.
\end{align*}
For $s\geq3$, the mapping $F$ is $C^\infty$ in the Fr\'{e}chet sense with respect to $(\omega_k,\varphi)$ with
\begin{align*}
&\partial_{\omega_k} F(\omega,\varphi)=\lambda(2p+1)((\omega\cdot\nabla)\varphi)^{2p}\partial_{\theta_k}\varphi,\quad k=1,2,\\
&\mathrm{D}F(\omega,\varphi)[h]=\lambda(2p+1)((\omega\cdot\nabla)\varphi)^{2p}(\omega\cdot\nabla)h.
\end{align*}
\end{lemm}
\begin{proof}
The  proof of the lemma is as seen in \cite{kosovalic2019self}.
\end{proof}

Linearizing equation \eqref{f:eq} about $\varphi=0$ at $\alpha=\gamma=0$ yields that
\begin{align*}
(\omega\cdot\nabla)^2\varphi+ \mu\Delta^2_{\nu}\varphi+m\varphi=0.
\end{align*}
In the Fourier basis $e^{\mathrm{i}j\cdot\theta}$, we have the following
\begin{align}\label{c4}
-(\omega_1j_1+\omega_2j_2)^2+ \mu(\nu^2_1j^{2}_1+\nu^2_2j^{2}_2)^2+m=0.
\end{align}
Since $m$ is  positive, $j=0$ is not a solution of equation \eqref{c4}. Moreover, due to the fact
\begin{align*}
\lim_{|j|\rightarrow\infty}-(\omega_1j_1+\omega_2j_2)^2+ \mu(\nu^2_1j^{2}_1+\nu^2_2j^{2}_2)^2+m=\infty,
\end{align*}
there exists no infinite $j\in\mathbb{Z}^2$ satisfying equation \eqref{c4}. The means that there is no small divisor problems at infinity. In other words, there are only finitely many Fourier modes satisfying equation \eqref{c4}. In particular, for fixed $j^*_k\in\mathbb{Z},k=1,2$, it is easy to see that
\begin{flalign*}
&\mathrm{(i)}\quad \text{$(\omega_{j^*_1},\omega_2,j^*_1,0)$ with free parameters $\omega_2$ and $\textstyle\omega_{j^*_1}^2=\frac{\mu\nu^4_1(j^*_1)^4+m}{(j^*_1)^2}$},&\\
&\mathrm{(ii)}\quad \text{$(\omega_1,\omega_{j^*_2},0,j^*_2)$ with free parameters $\omega_1$ and $\textstyle\omega_{j^*_2}^2=\frac{\mu\nu^4_2(j^*_2)^4+m}{(j^*_2)^2}$}&
\end{flalign*}
are solutions to  equation \eqref{c4}. From now on, we consider the ``critical frequency''
\begin{align*}
\omega_{j^*}=(\omega_{j^*_1},\omega_{j^*_2})\quad\text{with }\omega_{j^*_k}={\textstyle{\left(\frac{\mu\nu^4_k(j^*_k)^4+m}{(j^*_k)^2}\right)^{\frac12}}},k=1,2.
\end{align*}
We are now focused on solving equation \eqref{c4} with  $\omega=\omega_{j^*}$, that is,
\begin{align}\label{f:fixfrequency}
-(\omega_{j^*_1}j_1+\omega_{j^*_2}j_2)^2+\mu(\nu^{2}_1j^2_1+\nu^{2}_2j^2_2)^2+m=0.
\end{align}
In view of the expressions of $\omega_{j^*_k},k=1,2$, equation \eqref{f:fixfrequency} is equivalent to
\begin{align}\label{E:irrational}
2j_1j_2\frac{\omega_{j^*_1}}{\omega_{j^*_2}}-\frac{\mu\nu^{4}_1\left(j^4_1-(j^*_1)^2 j^2_1\right)+\mu\nu^{4}_2\left(j^4_2-(j^*_2)^2 j^2_2\right)+2\mu\nu^2_1\nu^2_2 j^2_1j^2_2+m(1-\frac{j^2_1}{(j^*_1)^2}-\frac{j^2_2}{(j^*_2)^2})}{\mu\nu^4_2(j^{*}_2)^2+\frac{m}{(j^*_2)^2}}=0.
\end{align}

For fixed positive numbers $\nu_1,\nu_2$, and fixed positive integers $j_1^*\neq \frac{\nu_2}{\nu_1}j^*_{2}$, we define the following set of parameters $(\mu,m)$ by
\begin{align*}
\textstyle\tilde{\mathcal{S}}_{\mu,m}:=\bigg\{(\mu,m)\in\mathbb{R}^{+}\times\mathbb{Q}^{+}:\mu\nu^4_1\in\mathbb{Q}^{+},
\frac{m}{\mu\nu^4_1}\notin\mathbb{N},\frac{m}{\mu\nu^4_2}\notin\mathbb{N},
\textstyle\left(\frac{\mu \nu^{4}_1(j^*_1)^4+m}{\mu \nu^{4}_2(j^*_{2})^4+m}\right)^{\frac12}
\in\mathbb{R}^+\backslash\mathbb{Q},&\nonumber\\
\textstyle\frac{\nu^2_1}{\nu^2_2}\left(\frac{\mu \nu^{4}_1(j^*_1)^4+m}{\mu \nu^{4}_2(j^*_{2})^4+m}\right)^{\frac12}
\in\mathbb{R}^+\backslash\mathbb{Q},\text{ for } \nu_1\in\mathbb{R}^+,\nu_2\in\mathbb{R}^+,\text{ with }{\nu^4_2}/{\nu^4_1}\in\mathbb{Q}^+\bigg\}&.
\end{align*}
Remark that $\tilde{\mathcal{S}}_{\mu,m}$ and $\mathcal{S}_{\mu,m}$ are two identical sets. Moreover, the relationship $\nu^4_2/\nu^4_1\in\mathbb{Q}^{+}$ between $\nu_1$ and $\nu_2$ shows that
\begin{align*}
\mu\nu^4_2=\mu\nu^4_1{\frac{\nu^4_2}{\nu^4_1}}\in\mathbb{Q}^{+}.
\end{align*}
\begin{lemm}\label{le:irrational}
For fixed positive numbers $\nu_1,\nu_2$ satisfying $\nu^4_2/\nu^4_1\in\mathbb{Q}^{+}$, and fixed positive integers $j_1^*\neq \frac{\nu_2}{\nu_1}j^*_{2}$, if $(\mu,m)$ belong to $\tilde{\mathcal{S}}_{\mu,m}$, then there exist only  four solutions
\begin{align*}
 j=(\pm j^*_1,0),\quad(0,\pm j^*_2)
\end{align*}
satisfying equation \eqref{f:fixfrequency}.
\end{lemm}
\begin{proof}
We first consider ${\nu^2_2}/{\nu^2_1}\in\mathbb{Q}^{+}$. It is clear that
\begin{align*}
\frac{\mu\nu^{4}_1\left(j^4_1-(j^*_1)^2 j^2_1\right)+\mu\nu^{4}_2\left(j^4_2-(j^*_2)^2 j^2_2\right)+2\mu\nu^4_1\frac{\nu^2_2}{\nu^2_1} j^2_1j^2_2+m(1-\frac{j^2_1}{(j^*_1)^2}-\frac{j^2_2}{(j^*_2)^2})}{\mu\nu^4_2(j^{*}_2)^2+\frac{m}{(j^*_2)^2}}\in\mathbb{Q}.
\end{align*}
Moreover, since $\omega_{j^*_1}/\omega_{j^*_2}$ is irrational, $\omega_{j^*_1}/\omega_{j^*_2},1$  are rationally independent. As a result,
\begin{align}\label{E:coefficient1}
&2j_1j_2=0,\\
&\textstyle\mu\nu^{4}_1\left(j^4_1-(j^*_1)^2 j^2_1\right)+\mu\nu^{4}_2\left(j^4_2-(j^*_2)^2 j^2_2\right)+2\mu\nu^2_1\nu^2_2 j^2_1j^2_2+m(1-\frac{j^2_1}{(j^*_1)^2}-\frac{j^2_2}{(j^*_2)^2})=0.\label{E:coefficient2}
\end{align}
Obviously, formula \eqref{E:coefficient1} implies that either $j_1=0$, or $j_2=0$.

For  $j_1=0$,  equality \eqref{E:coefficient2} can be reduced to
\begin{align*}
-\mu \nu^{4}_2(j^*_2)^2j^2_2\left((j^*_2)^2-j^2_2\right)+m\left((j^*_2)^2-j^2_2\right)=0,
\end{align*}
which then gives
\begin{align*}
(m-\mu \nu^{4}_2(j^*_2)^2j^2_2)\left((j^*_2)^2-j^2_2\right)=0.
\end{align*}
If $(j^*_2)^2-j^2_2\neq0$, then $m-\mu \nu^{4}_2(j^*_2)^2j^2_2=0$, which gives
\begin{align*}
(j^*_2)^2j^2_2=\frac{m}{\mu \nu^{4}_2}.
\end{align*}
This is impossible because of $\frac{m}{\mu \nu^{4}_2}\notin\mathbb{N}$. As a consequence,
\begin{align*}
(j^*_2)^2-j^2_2=0
\end{align*}
Thus $j_2=\pm j^*_2$.

If $j_2=0$, then we simplify  equality \eqref{E:coefficient2} to
\begin{align*}
(m-\mu \nu^{4}_1(j^*_1)^2j^2_1)\left((j^*_1)^2-j^2_1\right)=0.
\end{align*}
Since $\frac{m}{\mu \nu^{4}_1}\notin\mathbb{N}$, proceeding the similar technique as above yields that $j_1=\pm j^*_1$.

On the other hand, if  ${\nu^2_2}/{\nu^2_1}\in\mathbb{R}^+\backslash\mathbb{Q}$, then we cannot apply directly a similar argument as above due to the fact
\begin{align*}
\frac{\mu\nu^{4}_1\left(j^4_1-(j^*_1)^2 j^2_1\right)+\mu\nu^{4}_2\left(j^4_2-(j^*_2)^2 j^2_2\right)+2\mu\nu^4_1\frac{\nu^2_2}{\nu^2_1} j^2_1j^2_2+m(1-\frac{j^2_1}{(j^*_1)^2}-\frac{j^2_2}{(j^*_2)^2})}{\mu\nu^4_2(j^{*}_2)^2+\frac{m}{(j^*_2)^2}}
\in\mathbb{R}\backslash\mathbb{Q}.
\end{align*}
Since ${\nu^2_2}/{\nu^2_1}\in\mathbb{R}^+\backslash\mathbb{Q}$, we rewrite the equivalent equation \eqref{E:irrational} as
\begin{align}\label{E:aa}
2 j_1j_2\frac{\omega_{j^*_1}}{\omega_{j^*_2}}-\frac{2j^2_1j^2_2\mu\nu^4_1}{\mu\nu^4_2(j^{*}_2)^2+\frac{m}{(j^*_2)^2}}\frac{\nu^2_2}{\nu^2_1}-\frac{\mu\nu^{4}_1\left(j^4_1-(j^*_1)^2 j^2_1\right)+\mu\nu^{4}_2\left(j^4_2-(j^*_2)^2 j^2_2\right)+m(1-\frac{j^2_1}{(j^*_1)^2}-\frac{j^2_2}{(j^*_2)^2})}{\mu\nu^4_2(j^{*}_2)^2+\frac{m}{(j^*_2)^2}}\nonumber\\
=0.
\end{align}
Observe that
\begin{align*}
&2 j_1j_2\in\mathbb{Q},
\quad\frac{2j^2_1j^2_2\mu\nu^4_1}{\mu\nu^4_2(j^{*}_2)^2+\frac{m}{(j^*_2)^2}}\in \mathbb{Q},\\
&\frac{\mu\nu^{4}_1\left(j^4_1-(j^*_1)^2 j^2_1\right)+\mu\nu^{4}_2\left(j^4_2-(j^*_2)^2 j^2_2\right)+m(1-\frac{j^2_1}{(j^*_1)^2}-\frac{j^2_2}{(j^*_2)^2})}{\mu\nu^4_2(j^{*}_2)^2+\frac{m}{(j^*_2)^2}}\in\mathbb{Q}.
\end{align*}
Now we assert that $\frac{\omega_{j^*_1}}{\omega_{j^*_2}},\frac{\nu^2_2}{\nu^2_1},1$ are rationally independent.

Suppose that we could seek three rational numbers
$a,b$ and $c$, not all of which are zero, such that
\begin{align}\label{f:incommesurable}
a\frac{\omega_{j^*_1}}{\omega_{j^*_2}}+b\frac{\nu^2_2}{\nu^2_1}+c=0.
\end{align}
By squaring the above equation, we get
\begin{align*}
a^2\frac{\omega^2_{j^*_1}}{\omega^2_{j^*_2}}=b^2\frac{\nu^{4}_2}{\nu^4_1}+2bc\frac{\nu^2_2}{\nu^2_1}+c^2,
\end{align*}
that is,
\begin{align*}
2bc\frac{\nu^2_2}{\nu^2_1}=a^2\frac{\mu\nu^4_1(j^{*}_1)^2+\frac{m}{(j^*_1)^2}}{\mu\nu^4_2(j^{*}_2)^2+\frac{m}{(j^*_2)^2}}-b^2\frac{\nu^{4}_2}{\nu^4_1}-c^2.
\end{align*}
If $bc\neq0$, then
\begin{align*} \frac{\nu^2_2}{\nu^2_1}=\frac{a^2\frac{\mu\nu^4_1(j^{*}_1)^2+\frac{m}{(j^*_1)^2}}{\mu\nu^4_2(j^{*}_2)^2+\frac{m}{(j^*_2)^2}}-b^2\frac{\nu^{4}_2}{\nu^4_1}-c^2}{2bc}
\end{align*}
Since $\frac{\nu^2_2}{\nu^2_1}\in\mathbb{R}^+\backslash\mathbb{Q}, \frac{a^2\frac{\mu\nu^4_1(j^{*}_1)^2+\frac{m}{(j^*_1)^2}}{\mu\nu^4_2(j^{*}_2)^2+\frac{m}{(j^*_2)^2}}-b^2\frac{\nu^{4}_2}{\nu^4_1}-c^2}{2bc}\in\mathbb{Q}$, this leads to a contradiction. Hence either $b=0$, or $c=0$. In the first case equation \eqref{f:incommesurable} becomes
\begin{align*}
a\omega_{j^*_1}+c\omega_{j^*_2}=0.
\end{align*}
Hence it follows from $\left(\frac{\mu \nu^{4}_1(j^*_1)^4+m}{\mu \nu^{4}_2(j^*_{2})^4+m}\right)^{\frac12}
\in\mathbb{R}^+\backslash\mathbb{Q}$ that $a=c=0$. In the latter, using \eqref{f:incommesurable} yields that
\begin{align*}
a\frac{\omega_{j^*_1}}{\omega_{j^*_2}}+b\frac{\nu^2_2}{\nu^2_1}=0.
\end{align*}
According to the fact $\frac{\nu^2_1}{\nu^2_2}\left(\frac{\mu \nu^{4}_1(j^*_1)^4+m}{\mu \nu^{4}_2(j^*_{2})^4+m}\right)^{\frac12}\in\mathbb{R}^{+}\backslash\mathbb{Q}$, one has $a=b=0$.

By the above assertion, from \eqref{E:aa}, we obtain
\begin{align}\label{E:coefficient3}
&2 j_1j_2=0,\quad 2j^2_1j^2_2\mu\nu^4_1=0,\\
&\textstyle\mu\nu^{4}_1\left(j^4_1-(j^*_1)^2 j^2_1\right)+\mu\nu^{4}_2\left(j^4_2-(j^*_2)^2 j^2_2\right)+m(1-\frac{j^2_1}{(j^*_1)^2}-\frac{j^2_2}{(j^*_2)^2})=0.\label{E:coefficient4}
\end{align}
Two equalities in \eqref{E:coefficient3} imply that either $j_1=0$, or $j_2=0$. In the first case, equality \eqref{E:coefficient4} can be simplified to
\begin{align*}
(m-\mu \nu^{4}_2(j^*_2)^2j^2_2)\left((j^*_2)^2-j^2_2\right)=0.
\end{align*}
In the latter
\begin{align*}
(m-\mu \nu^{4}_1(j^*_1)^2j^2_1)\left((j^*_1)^2-j^2_1\right)=0.
\end{align*}
By a similar argument as shown in the case ${\nu^2_2}/{\nu^2_1}\in\mathbb{Q}^{+}$, we derive that  $j=(\pm j^*_1,0),(0,\pm j^*_2)$ are solutions of equation \eqref{f:fixfrequency}.

As a result,  we arrive at the conclusion of the lemma.
\end{proof}
Now we make some remarks on the set $\tilde{\mathcal{S}}_{\mu,m}$ of parameters $(\mu,m)$.
\begin{rema}\label{remark1}
In fact, the relationship $\nu^4_2/\nu^4_1\in\mathbb{Q}^{+}$ between $\nu_1$ and $\nu_2$ means that either $\nu^4_1\in\mathbb{Q}^{+},\nu^4_2\in\mathbb{Q}^{+}$, or $\nu^4_1\in\mathbb{R}^+\backslash\mathbb{Q},\nu^4_2\in\mathbb{R}^+\backslash\mathbb{Q}$. Moreover, since
\begin{align*}
\nu^4_2/\nu^4_1\in\mathbb{Q}^{+},\mu\nu^4_1\in\mathbb{Q}^{+}\Longleftrightarrow\mu\nu^4_1\in\mathbb{Q}^{+},\mu\nu^4_2\in\mathbb{Q}^{+},
\end{align*}
one has
\begin{flalign*}
&\text{$\mathrm{(i)}$\quad If $\nu^4_1\in\mathbb{Q}^{+},\nu^4_2\in\mathbb{Q}^{+}$,  then $\mu\in\mathbb{Q}^{+}$};&\\
&\text{$\mathrm{(ii)}$\quad If $\nu^4_1\in\mathbb{R}^+\backslash\mathbb{Q},\nu^4_2\in\mathbb{R}^+\backslash\mathbb{Q}$,  then $\mu\in\mathbb{R}^+\backslash\mathbb{Q}$}.&
\end{flalign*}
More precisely, the set $\tilde{\mathcal{S}}_{\mu,m}$ of parameters $(\mu,m)$ can be clearly expressed as the following several cases  by the range of $\nu_1,\nu_2$.

On the one hand, for $\nu^4_1\in\mathbb{Q}^{+}$ and $\nu^4_2\in\mathbb{Q}^{+}$, the set $\tilde{\mathcal{S}}_{\mu,m}$ is equivalent to the following four cases.

\vspace{0.3cm}
\textit{\underline{\textbf{Case 1}}}: $\nu^{4}_1\in\mathbb{Q}^+,\nu^{4}_2\in\mathbb{Q}^+$, with either $\nu^{2}_1\in\mathbb{Q}^+,\nu^{2}_2\in\mathbb{Q}^+$,
or $\nu^{2}_1\in\mathbb{R}^+\backslash\mathbb{Q},\nu^{2}_2\in\mathbb{R}^+\backslash\mathbb{Q}$ and $\nu^{2}_2/\nu^{2}_1\in\mathbb{Q}^+$.
For fixed positive integers $j_1^*\neq \frac{\nu_2}{\nu_1}j^*_{2}$, we express the set of parameters $(\mu,m)$  as follows
\begin{align*}
\textstyle\tilde{\mathcal{S}}_{\mu,m,1}:=\bigg\{(\mu,m)\in\mathbb{Q}^{+}\times\mathbb{Q}^+:\frac{m}{\mu\nu^4_1}\notin\mathbb{N},\frac{m}{\mu\nu^4_2}\notin\mathbb{N},\left(\frac{\mu \nu^{4}_1(j^*_1)^4+m}{\mu \nu^{4}_2(j^*_{2})^4+m}\right)^{\frac12}
\in\mathbb{R}^+\backslash\mathbb{Q}\bigg\}.
\end{align*}

\textit{\underline{\textbf{Case 2}}}: $\nu^{4}_1\in\mathbb{Q}^+,\nu^{4}_2\in\mathbb{Q}^+$, with $\nu^{2}_1\in\mathbb{R}^+\backslash\mathbb{Q}$, $\nu^{2}_2\in\mathbb{Q}^+$. For fixed positive integers $j_1^*\neq \frac{\nu_2}{\nu_1}j^*_{2}$, the set of parameters $(\mu,m)$ can be written  as follows
\begin{align*}
\textstyle\tilde{\mathcal{S}}_{\mu,m,2}:=\bigg\{(\mu,m)\in\mathbb{Q}^{+}\times\mathbb{Q}^+:\frac{m}{\mu\nu^4_1}\notin\mathbb{N},\frac{m}{\mu\nu^4_2}\notin\mathbb{N},\left(\frac{\mu \nu^{4}_1(j^*_1)^4+m}{\mu \nu^{4}_2(j^*_{2})^4+m}\right)^{\frac12}
\in\mathbb{R}^+\backslash\mathbb{Q},&\\
\textstyle\frac{1}{\nu^2_1}\left(\frac{\mu \nu^{4}_1(j^*_1)^4+m}{\mu \nu^{4}_2(j^*_{2})^4+m}\right)^{\frac12}
\in\mathbb{R}^+\backslash\mathbb{Q}\bigg\}&.
\end{align*}

\textit{\underline{\textbf{Case 3}}}: $\nu^{4}_1\in\mathbb{Q}^+,\nu^{4}_2\in\mathbb{Q}^+$, with $\nu^{2}_1\in\mathbb{Q}^+$, $\nu^{2}_2\in\mathbb{R}^+\backslash\mathbb{Q}$. For fixed positive integers $j_1^*\neq \frac{\nu_2}{\nu_1}j^*_{2}$, the set of parameters $(\mu,m)$ can be expressed  as follows
\begin{align*}
\textstyle\tilde{\mathcal{S}}_{\mu,m,3}:=\bigg\{(\mu,m)\in\mathbb{Q}^{+}\times\mathbb{Q}^+:\frac{m}{\mu\nu^4_1}\notin\mathbb{N},
\frac{m}{\mu\nu^4_2}\notin\mathbb{N},\left(\frac{\mu \nu^{4}_1(j^*_1)^4+m}{\mu \nu^{4}_2(j^*_{2})^4+m}\right)^{\frac12}
\in\mathbb{R}^+\backslash\mathbb{Q},&\\
\textstyle\frac{1}{\nu^2_2}\left(\frac{\mu \nu^{4}_1(j^*_1)^4+m}{\mu \nu^{4}_2(j^*_{2})^4+m}\right)^{\frac12}
\in\mathbb{R}^+\backslash\mathbb{Q}\bigg\}&.
\end{align*}

\textit{\underline{\textbf{Case 4}}}: $\nu^{4}_1\in\mathbb{Q}^+,\nu^{4}_2\in\mathbb{Q}^+$, with $\nu^{2}_1\in\mathbb{R}^+\backslash\mathbb{Q}$, $\nu^{2}_2\in\mathbb{R}^+\backslash\mathbb{Q}$ and $\nu^{2}_2/\nu^{2}_1\in\mathbb{R}^+\backslash\mathbb{Q}$.
For fixed positive integers $j_1^*\neq \frac{\nu_2}{\nu_1}j^*_{2}$, we write the set of parameters $(\mu,m)$ as follows
\begin{align*}
\textstyle\tilde{\mathcal{S}}_{\mu,m,4}:=\bigg\{(\mu,m)\in\mathbb{Q}^{+}\times\mathbb{Q}^+:\frac{m}{\mu\nu^4_1}\notin\mathbb{N},\frac{m}{\mu\nu^4_2}\notin\mathbb{N},\left(\frac{\mu \nu^{4}_1(j^*_1)^4+m}{\mu \nu^{4}_2(j^*_{2})^4+m}\right)^{\frac12}
\in\mathbb{R}^+\backslash\mathbb{Q},&\\
\textstyle\frac{\nu^2_1}{\nu^2_2}\left(\frac{\mu \nu^{4}_1(j^*_1)^4+m}{\mu \nu^{4}_2(j^*_{2})^4+m}\right)^{\frac12}
\in\mathbb{R}^+\backslash\mathbb{Q}\bigg\}&.
\end{align*}

On the other hand, for $\nu^4_1\in\mathbb{R}^+\backslash\mathbb{Q}$ and $\nu^4_2\in\mathbb{R}^+\backslash\mathbb{Q}$, the set $\tilde{\mathcal{S}}_{\mu,m}$ is equivalent to the following two cases.

\vspace{0.3cm}
\textit{\underline{\textbf{Case 5}}}: $\nu^4_2/\nu^4_1\in\mathbb{Q}^{+}$, with $\nu^{4}_1\in\mathbb{R}^+\backslash\mathbb{Q},\nu^{4}_2\in\mathbb{R}^+\backslash\mathbb{Q}$ and ${\nu^2_2}/{\nu^2_1}\in\mathbb{Q}^{+}$. For fixed positive integers $j_1^*\neq \frac{\nu_2}{\nu_1}j^*_{2}$,  the set of parameters $(\mu,m)$ can be expressed  as follows
\begin{align*}
\textstyle\tilde{\mathcal{S}}_{\mu,m,5}:=\bigg\{(\mu,m)\in(\mathbb{R}^{+}\backslash\mathbb{Q})\times\mathbb{Q}^+:\mu\nu^4_1\in\mathbb{Q}^{+},\frac{m}{\mu\nu^4_1}\notin\mathbb{N},\frac{m}{\mu\nu^4_2}\notin\mathbb{N},
\left(\frac{\mu \nu^{4}_1(j^*_1)^4+m}{\mu \nu^{4}_2(j^*_{2})^4+m}\right)^{\frac12}
\in\mathbb{R}^+\backslash\mathbb{Q}\bigg\}.
\end{align*}

\textit{\underline{\textbf{Case 6}}}: $\nu^4_2/\nu^4_1\in\mathbb{Q}^{+}$, with $\nu^{4}_1\in\mathbb{R}^+\backslash\mathbb{Q},\nu^{4}_2\in\mathbb{R}^+\backslash\mathbb{Q}$ and ${\nu^2_2}/{\nu^2_1}\in\mathbb{R}^{+}\backslash\mathbb{Q}$. For fixed positive integers $j_1^*\neq \frac{\nu_2}{\nu_1}j^*_{2}$,  the set of parameters $(\mu,m)$ can be written as follows
\begin{align*}
\textstyle\tilde{\mathcal{S}}_{\mu,m,6}:=\bigg\{(\mu,m)\in(\mathbb{R}^{+}\backslash\mathbb{Q})\times\mathbb{Q}^+:
\mu\nu^4_1\in\mathbb{Q}^{+},\frac{m}{\mu\nu^4_1}\notin\mathbb{N},\frac{m}{\mu\nu^4_2}\notin\mathbb{N},\left(\frac{\mu \nu^{4}_1(j^*_1)^4+m}{\mu \nu^{4}_2(j^*_{2})^4+m}\right)^{\frac12}
\in\mathbb{R}^+\backslash\mathbb{Q},&\\
\textstyle\frac{\nu^2_1}{\nu^2_2}\left(\frac{\mu \nu^{4}_1(j^*_1)^4+m}{\mu \nu^{4}_2(j^*_{2})^4+m}\right)^{\frac12}
\in\mathbb{R}^+\backslash\mathbb{Q}\bigg\}&.
\end{align*}
\end{rema}
The following lemma corresponds to the density of the set $\tilde{\mathcal{S}}_{\mu,m}$ in the  space $\mathbb{R}^+\times\mathbb{R}^+$.

\begin{lemm}\label{le:density1}
For fixed positive numbers $\nu_1,\nu_2$ satisfying $\nu^4_2/\nu^4_1\in\mathbb{Q}^{+},$
and fixed positive integers $j_1^*\neq \frac{\nu_2}{\nu_1}j^*_{2}$, the set $\tilde{\mathcal{S}}_{\mu,m}$ is dense in the  space $\mathbb{R}^+\times\mathbb{R}^+$.
\end{lemm}
\begin{proof}
Observe that the following set
\begin{align*}
\mathcal{Q}:=\left\{q\in\mathbb{Q}^{+}:(q)^{\frac12}\in\mathbb{R}^{+}\backslash \mathbb{Q}\right\}
\end{align*}
is dense in $\mathbb{R}^{+}$ by using the density of the set of rational numbers composed by ratio of distinct primes. Let $({\mu}_0,m_0)\in\mathbb{R}^{+}\times\mathbb{R}^+$. For all $\epsilon>0$, we want to look for $(\tilde{\mu},\tilde{m})\in \tilde{\mathcal{S}}_{\mu,m}\cap\mathcal{B}_{\epsilon}(\mu_0,m_0)$, where
\begin{align*}
\mathcal{B}_{\epsilon}(\mu_0,m_0):=\{(\mu,m):|\mu-\mu_0|<\epsilon,|m-m_0|<\epsilon\}.
\end{align*}

For fixed $\nu_1>0,\nu_2>0$, we first consider ${\nu^2_2}/{\nu^2_1}\in\mathbb{Q}^{+}$. Clearly, for $\nu^4_1\in\mathbb{Q}^{+}$, there is $\tilde{\mu}\in(\mu_0-\epsilon,\mu_0+\epsilon)\cap\mathbb{Q}^+$ such that $\tilde{\mu}\nu^4_1\in\mathbb{Q}^{+}$. If $\nu^4_1\in\mathbb{R}^+\backslash\mathbb{Q}$,
then we can find $\tilde{\mu}\in(\mu_0-\epsilon,\mu_0+\epsilon)\cap(\mathbb{R}^+\backslash\mathbb{Q})$ satisfying $\tilde{\mu}\nu^4_1\in\mathbb{Q}^{+}$.
From the fact $\nu^4_2/\nu^4_1\in\mathbb{Q}^{+}$, it follows that $\tilde{\mu}\nu^4_2\in\mathbb{Q}^{+}$.
Hence we can look for $m'\in(m_0-\epsilon,m_0+\epsilon)$ with $\frac{m'}{\tilde{\mu}\nu^4_1}\notin\mathbb{N},
\frac{m'}{\tilde{\mu}\nu^4_2}\notin\mathbb{N}$. As a result, there exists an open interval $\mathcal{U}$ contained in $(m_0-\epsilon,m_0+\epsilon)$ satisfying $\frac{m}{\tilde{\mu}\nu^4_1}\notin\mathbb{N},
\frac{m}{\tilde{\mu}\nu^4_2}\notin\mathbb{N},\forall m\in\mathcal{U}$. For fixed $\nu_1,\nu_2,j^*_1,j^*_2$ and fixed $\tilde{\mu}\in(\mu_0-\epsilon,\mu_0+\epsilon)$, we further introduce a mapping as follows
\begin{align*}
\textstyle\mathrm{g}:\mathcal{U}\longrightarrow\mathbb{R}^+,\quad
m\longmapsto{{\frac{\tilde{\mu}\nu^4_1(j^*_1)^4+m}{\tilde{\mu}\nu^4_2(j^*_{2})^4+m}}}.
\end{align*}
Since $j_1^*\neq \frac{\nu_2}{\nu_1}j^*_{2}$, the function $\mathrm{g}$ is not a constant. In view of  the intermediate value theorem together with the density of $\mathcal Q$, there is $\tilde{m}\in\mathcal{U}$ such that
\begin{align*}
\mathrm{g}(\tilde{m})\in\mathbb{Q}^+,\quad (\mathrm{g}(\tilde{m}))^{\frac{1}{2}}\in\mathbb{R}^+\backslash\mathbb{Q},
\end{align*}
which leads to  $\tilde{m}\in\mathbb{Q}^+$. Therefore, $(\tilde{\mu},\tilde{m})\in \tilde{\mathcal{S}}_{\mu,m}\cap\mathcal{B}_{\epsilon}(\mu_0,m_0)$.

On the other hand, let ${\nu^2_2}/{\nu^2_1}\in\mathbb{R}^+\backslash\mathbb{Q}$ for fixed $\nu_1>0,\nu_2>0$. Denote \begin{align*}
&\textstyle\tilde{\mathcal{S}}^1_{\mu,m}:=\bigg\{(\mu,m)\in\mathbb{R}^{+}\times\mathbb{Q}^{+}:\mu\nu^4_1\in\mathbb{Q}^{+},
\frac{m}{\mu\nu^4_1}\notin\mathbb{N},\frac{m}{\mu\nu^4_2}\notin\mathbb{N},
\textstyle\left(\frac{\mu \nu^{4}_1(j^*_1)^4+m}{\mu \nu^{4}_2(j^*_{2})^4+m}\right)^{\frac12}
\in\mathbb{R}^+\backslash\mathbb{Q}\bigg\},\\
&\textstyle\tilde{\mathcal{S}}^2_{\mu,m}:=\bigg\{(\mu,m)\in\mathbb{R}^{+}\times\mathbb{Q}^{+}:\mu\nu^4_1\in\mathbb{Q}^{+},
\frac{m}{\mu\nu^4_1}\notin\mathbb{N},\frac{m}{\mu\nu^4_2}\notin\mathbb{N},\textstyle\frac{\nu^2_1}{\nu^2_2}\left(\frac{\mu \nu^{4}_1(j^*_1)^4+m}{\mu \nu^{4}_2(j^*_{2})^4+m}\right)^{\frac12}
\in\mathbb{R}^+\backslash\mathbb{Q}\bigg\}.
\end{align*}
By a similar argument as above, the sets $\tilde{\mathcal{S}}^1_{\mu,m},\tilde{\mathcal{S}}^2_{\mu,m}$ are dense in the  space $\mathbb{R}^+\times\mathbb{R}^+$. Then $\tilde{\mathcal{S}}^1_{\mu,m}\cap\tilde{\mathcal{S}}^2_{\mu,m}$ is dense in the  space $\mathbb{R}^+\times\mathbb{R}^+$, that is, we obtain the density of $\tilde{\mathcal{S}}_{\mu,m}$  in the  space $\mathbb{R}^+\times\mathbb{R}^+$.

Hence we complete the proof of the lemma.
\end{proof}

The rational coefficients $\mu\nu^4_1,\mu\nu^4_2,m$ play key roles in the proof of Lemma \ref{le:irrational}.  Notice that
\begin{flalign*}
&\text{$\mathrm{(i)}$\quad If either $\mu\nu^4_1\in\mathbb{R}^+\backslash\mathbb{Q},\mu\nu^4_2\in\mathbb{Q}^{+}$, or $\mu\nu^4_1\in\mathbb{Q}^+,\mu\nu^4_2\in\mathbb{R}^+\backslash\mathbb{Q}$,  then $\mu\nu^2_1\nu^2_2\in\mathbb{R}^+\backslash\mathbb{Q}$};&\\
&\text{$\mathrm{(ii)}$\quad If $\mu\nu^4_1\in\mathbb{R}^+\backslash\mathbb{Q},\mu\nu^4_2\in\mathbb{R}^+\backslash\mathbb{Q}$,  then either $\mu\nu^2_1\nu^2_2\in\mathbb{Q}^+$, or $\mu\nu^2_1\nu^2_2\in\mathbb{R}^+\backslash\mathbb{Q}$}.&
\end{flalign*}
As a result, at least two of $\mu\nu^4_1,\mu\nu^4_2,\mu\nu^2_1\nu^2_2$ are irrational numbers. We just consider
\begin{align*}
&(\mu,m)\in\mathbb{R}^+\times\mathbb{R}^{+},\text{ with }
\mu^2\in\mathbb{Q}^+,m^2\in\mathbb{Q}^+,\textstyle\frac{m}{\mu}\in\mathbb{R}^{+}\backslash\mathbb{Q},\\
&\nu^{4}_1\in\mathbb{Q}^+,\nu^{4}_2\in\mathbb{Q}^+ \text{ and } \nu^{2}_2/\nu^{2}_1\in\mathbb{Q}^+.
\end{align*}
For fixed positive numbers $\nu_1,\nu_2$, and fixed positive integers $j_1^*\neq \frac{\nu_2}{\nu_1}j^*_{2}$, we define the following set of parameters $(\mu,m)$ by
\begin{align*}
\textstyle\tilde{\mathcal{S}}'_{\mu,m}:=\bigg\{(\mu,m)\in\mathbb{R}^+\times\mathbb{R}^{+}:
\mu^2\in\mathbb{Q}^+,m^2\in\mathbb{Q}^+,\frac{m}{\mu}\in\mathbb{R}^{+}\backslash\mathbb{Q}, \text{ for } \nu_1\in\mathbb{R}^+,\nu_2\in\mathbb{R}^+, \\
\text{ with } \nu^{4}_1\in\mathbb{Q}^+,\nu^{4}_2\in\mathbb{Q}^+
\text{ and } \nu^{2}_2/\nu^{2}_1\in\mathbb{Q}^+\bigg\}&.
\end{align*}
Remark that $\tilde{\mathcal{S}}'_{\mu,m}$ and $\mathcal{S}'_{\mu,m}$ are two identical sets. Obviously,
\begin{align*}
\textstyle\left(\frac{\mu \nu^{4}_1(j^*_1)^4+m}{\mu \nu^{4}_2(j^*_{2})^4+m}\right)^{\frac12}
\in\mathbb{R}^+\backslash\mathbb{Q},\quad \forall(\mu,m)\in\tilde{\mathcal{S}}'_{\mu,m}.
\end{align*}

Let us now solve equation \eqref{f:fixfrequency} for $(\mu,m)\in\tilde{\mathcal{S}}'_{\mu,m}$.
\begin{lemm}\label{le:irrational11}
For fixed positive numbers $\nu_1,\nu_2$ satisfying $\nu^{4}_1\in\mathbb{Q}^+,\nu^{4}_2\in\mathbb{Q}^+
\text{ and } \nu^{2}_2/\nu^{2}_1\in\mathbb{Q}^+$, and fixed positive integers $j_1^*\neq \frac{\nu_2}{\nu_1}j^*_{2}$, if $(\mu,m)$ are in $\tilde{\mathcal{S}}'_{\mu,m}$, then equation \eqref{f:fixfrequency} has only  four solutions
\begin{align*}
 j=(\pm j^*_1,0),\quad(0,\pm j^*_2).
\end{align*}
\end{lemm}
\begin{proof}
Because of the expressions of $\omega_{j^*_k},k=1,2$, squaring equation \eqref{f:fixfrequency} yields that
\begin{align*}
4j^2_1j^2_2\omega^2_{j^*_1}\omega^2_{j^*_2}=&\textstyle\left(\nu^{4}_1\left(j^4_1-(j^*_1)^2 j^2_1\right)+\nu^{4}_2\left(j^4_2-(j^*_2)^2 j^2_2\right)+2\nu^2_1\nu^2_2j^2_1j^2_2\right)^2\mu^2+m^2(1-\frac{j^2_1}{(j^*_1)^2}-\frac{j^2_2}{(j^*_2)^2})^2\\
&\textstyle+2\left(\nu^{4}_1\left(j^4_1-(j^*_1)^2 j^2_1\right)+\nu^{4}_2\left(j^4_2-(j^*_2)^2 j^2_2\right)+2\nu^2_1\nu^2_2j^2_1j^2_2\right)(1-\frac{j^2_1}{(j^*_1)^2}-\frac{j^2_2}{(j^*_2)^2})m\mu,
\end{align*}
which leads to
\begin{align*}
&\textstyle 4j^2_1j^2_2\nu^4_1\nu^4_2(j^*_1)^2(j^*_2)^2\mu^2+\frac{4m^2j^2_1j^2_2}{(j^*_1)^2(j^*_2)^2}+\frac{4j^2_1j^2_2( \nu^4_1(j^*_1)^4+\nu^4_2(j^*_2)^4)}{(j^*_1)^2(j^*_2)^2}m\mu\\
=&\textstyle\left(\nu^{4}_1\left(j^4_1-(j^*_1)^2 j^2_1\right)+\nu^{4}_2\left(j^4_2-(j^*_2)^2 j^2_2\right)+2\nu^2_1\nu^2_2j^2_1j^2_2\right)^2\mu^2+m^2(1-\frac{j^2_1}{(j^*_1)^2}-\frac{j^2_2}{(j^*_2)^2})^2\\
&\textstyle+2\left(\nu^{4}_1\left(j^4_1-(j^*_1)^2 j^2_1\right)+\nu^{4}_2\left(j^4_2-(j^*_2)^2 j^2_2\right)+2\nu^2_1\nu^2_2j^2_1j^2_2\right)(1-\frac{j^2_1}{(j^*_1)^2}-\frac{j^2_2}{(j^*_2)^2})m\mu.
\end{align*}
Therefore,
\begin{align*}
&\textstyle(\frac{4j^2_1j^2_2( \nu^4_1(j^*_1)^4+\nu^4_2(j^*_2)^4)}{(j^*_1)^2(j^*_2)^2}-2\left(\nu^{4}_1\left(j^4_1-(j^*_1)^2 j^2_1\right)+\nu^{4}_2\left(j^4_2-(j^*_2)^2 j^2_2\right)+2\nu^2_1\nu^2_2j^2_1j^2_2\right)(1-\frac{j^2_1}{(j^*_1)^2}-\frac{j^2_2}{(j^*_2)^2}))\mu^2\frac{m}{\mu}\\
&\textstyle+4j^2_1j^2_2\nu^4_1\nu^4_2(j^*_1)^2(j^*_2)^2\mu^2+\frac{4m^2j^2_1j^2_2}{(j^*_1)^2(j^*_2)^2}-\left(\nu^{4}_1\left(j^4_1-(j^*_1)^2 j^2_1\right)+\nu^{4}_2\left(j^4_2-(j^*_2)^2 j^2_2\right)+2\nu^2_1\nu^2_2j^2_1j^2_2\right)^2\mu^2\\
&\textstyle-m^2(1-\frac{j^2_1}{(j^*_1)^2}-\frac{j^2_2}{(j^*_2)^2})^2=0.
\end{align*}
For $\nu^{4}_1\in\mathbb{Q}^+,\nu^{4}_2\in\mathbb{Q}^+
\text{ and } \nu^{2}_2/\nu^{2}_1\in\mathbb{Q}^+$, it follows from $
\mu^2\in\mathbb{Q}^+,m^2\in\mathbb{Q}^+,\textstyle\frac{m}{\mu}\in\mathbb{R}^{+}\backslash\mathbb{Q}$ that
\begin{align*}
&\textstyle2\left(\nu^{4}_1\left(j^4_1-(j^*_1)^2 j^2_1\right)+\nu^{4}_2\left(j^4_2-(j^*_2)^2 j^2_2\right)+2\nu^2_1\nu^2_2j^2_1j^2_2\right)(1-\frac{j^2_1}{(j^*_1)^2}-\frac{j^2_2}{(j^*_2)^2})\\
&\textstyle-\frac{4j^2_1j^2_2( \nu^4_1(j^*_1)^4+\nu^4_2(j^*_2)^4)}{(j^*_1)^2(j^*_2)^2}=0,\\
&\textstyle4j^2_1j^2_2\nu^4_1\nu^4_2(j^*_1)^2(j^*_2)^2\mu^2+\frac{4m^2j^2_1j^2_2}{(j^*_1)^2(j^*_2)^2}-\left(\nu^{4}_1\left(j^4_1-(j^*_1)^2 j^2_1\right)+\nu^{4}_2\left(j^4_2-(j^*_2)^2 j^2_2\right)+2\nu^2_1\nu^2_2j^2_1j^2_2\right)^2\mu^2\nonumber\\
&\textstyle-m^2(1-\frac{j^2_1}{(j^*_1)^2}-\frac{j^2_2}{(j^*_2)^2})^2=0,
\end{align*}
that is,
\begin{align}
&\nu^4_1(j^*_2)^2j^6_1+\nu^4_2(j^*_1)^2j^6_2+(\nu^4_1(j^*_1)^2+2\nu_1^2\nu^2_2(j^*_2)^2)j^4_1j^2_2
+(\nu^4_2(j^*_2)^2+2\nu^2_1\nu^2_2(j^*_1)^2)j^2_1j^4_2\nonumber\\
&-2\nu^4_1(j^*_1)^2(j^*_2)^2j^4_1
-2\nu^4_2(j^*_1)^2(j^*_2)^2j^4_2+(\nu^2_1(j^*_1)^2-\nu^2_2(j^*_2)^2)^2j^2_1j^2_2\nonumber\\
&+\nu^4_1(j^*_1)^4(j^*_2)^2j^2_1+\nu^4_2(j^*_1)^2(j^*_2)^4j^2_2=0,\label{E:diophantine}\\
&\mu^2\nu^8_1(j^*_1)^2(j^*_2)^2j^8_1+\mu^2\nu^8_2(j^*_1)^2(j^*_2)^2j^8_2+6\mu^2\nu^4_1\nu^4_2(j^*_1)^2(j^*_2)^2j^4_1j^4_2
+2\mu^2\nu^6_1\nu^2_2(j^*_1)^2(j^*_2)^2j^6_1j^2_2\nonumber\\
&+2\mu^2\nu^2_1\nu^6_2(j^*_1)^2(j^*_2)^2j^2_1j^6_2-2\mu^2\nu^8_1(j^*_1)^4(j^*_2)^2j^6_1-2\mu^2\nu^8_2(j^*_1)^2(j^*_2)^4j^6_2\nonumber\\
&-2\mu^2(\nu^4_1\nu^4_2(j^*_1)^2(j^*_2)^4+\nu^6_1\nu^2_2(j^*_1)^4(j^*_2)^2)j^4_1j^2_2
-2\mu^2(\nu^4_1\nu^4_2(j^*_1)^4(j^*_2)^2+\nu^2_1\nu^6_2(j^*_1)^2(j^*_2)^4)j^2_1j^4_2\nonumber\\
&+\mu^2\nu^8_1(j^*_1)^6(j^*_2)^2j^4_1+\mu^2\nu^8_2(j^*_1)^2(j^*_2)^6j^4_2
-2\mu^2\nu^4_1\nu^4_2(j^*_1)^4(j^*_2)^4j^2_1j^2_2-4m^2j^2_1j^2_2\nonumber\\
&-m^2(j^*_2)^2j^2_1
-m^2(j^*_1)^2j^2_2+m^2(j^*_1)^2(j^*_2)^2=0,\label{E:diophantine1}
\end{align}
which are two Diophantine equations. Observe that
\begin{align*}
\nu^4_1(j^*_2)^2j^6_1+\nu^4_1(j^*_1)^4(j^*_2)^2j^2_1\geq2\nu^4_1(j^*_1)^2(j^*_2)^2j^4_1,\\
\nu^4_2(j^*_1)^2j^6_2+\nu^4_2(j^*_1)^2(j^*_2)^4j^2_2\geq2\nu^4_2(j^*_1)^2(j^*_2)^2j^4_2.
\end{align*}
Combining this with \eqref{E:diophantine} yields that
\begin{align*}
0\geq(\nu^4_1(j^*_1)^2+2\nu_1^2\nu^2_2(j^*_2)^2)j^4_1j^2_2
+(\nu^4_2(j^*_2)^2+2\nu^2_1\nu^2_2(j^*_1)^2)j^2_1j^4_2+(\nu^2_1(j^*_1)^2-\nu^2_2(j^*_2)^2)^2j^2_1j^2_2\geq0
\end{align*}
Therefore,
\begin{align*}
&(\nu^4_1(j^*_1)^2+2\nu_1^2\nu^2_2(j^*_2)^2)j^4_1j^2_2
+(\nu^4_2(j^*_2)^2+2\nu^2_1\nu^2_2(j^*_1)^2)j^2_1j^4_2+(\nu^2_1(j^*_1)^2-\nu^2_2(j^*_2)^2)^2j^2_1j^2_2\\
&=((\nu^4_1(j^*_1)^2+2\nu_1^2\nu^2_2(j^*_2)^2)j^2_1
+(\nu^4_2(j^*_2)^2+2\nu^2_1\nu^2_2(j^*_1)^2)j^2_2+(\nu^2_1(j^*_1)^2-\nu^2_2(j^*_2)^2)^2)j^2_1j^2_2\\
&=0.
\end{align*}
We obtain either $j_1=0$, or $j_2=0$. For $j_1=0$, from \eqref{E:diophantine}, we have
\begin{align*}
\nu^4_2(j^*_1)^2j^2_2(j_2+j^*_2)^2(j_2-j^*_2)^2=0,
\end{align*}
which leads to either $j_2=0$, or $ j_2=\pm j^*_2$. On the other hand, if $j_2=0$, we conclude either $j_1=0$, or $ j_1=\pm j^*_1$. As a consequence, equation \eqref{E:diophantine} admits five solutions
\begin{align*}
(0,0),\quad (\pm j^*_1,0),\quad (0,\pm j^*_2),.
\end{align*}
By substituting these solutions into equation  \eqref{E:diophantine1}, equations  \eqref{E:diophantine}--\eqref{E:diophantine1} has solutions $ (\pm j^*_1,0),(0,\pm j^*_2)$,
meaning that equation \eqref{f:fixfrequency} has only  four solutions $(\pm j^*_1,0),(0,\pm j^*_2)$.

The proof of the lemma is now completed.
\end{proof}

We also verify the density of the set $\tilde{\mathcal{S}}'_{\mu,m}$ in the  space $\mathbb{R}^{+}\times\mathbb{R}^{+}$.
\begin{lemm}\label{le:density2}
For fixed positive numbers $\nu_1,\nu_2$ satisfying $\nu^{4}_1\in\mathbb{Q}^+,\nu^{4}_2\in\mathbb{Q}^+$ and $\nu^{2}_2/\nu^{2}_1\in\mathbb{Q}^+$, and fixed positive integers $j_1^*\neq \frac{\nu_2}{\nu_1}j^*_{2}$, the set $\tilde{\mathcal{S}}'_{\mu,m}$ is dense in the  space $\mathbb{R}^+\times\mathbb{R}^+$.
\end{lemm}
In next analysis, we will consider the critical cases in which $(\omega,\alpha,\gamma):=(\omega_{j^*},0,0)$. We  further set
\begin{align*}
&J=\left\{j\in\mathbb{Z}^2:j\neq(\pm j^*_1,0),(0,\pm j^*_2)\right\},\\
&J^\bot=\left\{j\in\mathbb{Z}^2:j=(\pm j^*_1,0),(0,\pm j^*_2)\right\}.
\end{align*}
Denote by $V,W$ the kernel space of the operator $L_{\omega_{j^*},0,0}$ and its orthogonal complement in $H^0$, respectively. Then the corresponding projection operators are defined as
\begin{align*}
\Pi_{V}:H^s\longrightarrow V,\quad\Pi_{W}:H^s\longrightarrow W.
\end{align*}

First, the following fact follows from Lemma \ref{le:irrational}.
\begin{lemm}\label{kernel}
The space $V$ is 4-dimensional with 
\begin{align*}
\textstyle V=\left\{v=\sum_{j\in J^\bot}v_je^{\mathrm{i}j\cdot \theta}\in H^0\right\}.
\end{align*}
\end{lemm}

Therefore the corresponding space $W$ can be written as
\begin{align*}
\textstyle W=\left\{ w=\sum_{j\in J}w_{j}e^{\mathrm{i}j\cdot\theta}\in H^{0}\right\}.
\end{align*}
Obviously, the space $H^s$ is decomposed as the direct sum of $V\cap H^s$ and $W\cap H^s$. For every $\varphi\in H^s$, we can write $\varphi=v+w$, where $v\in V\cap H^s$ and $w\in W\cap H^s$. By implementing  the Lyapunov--Schmidt reduction with respect to the above decomposition, equation \eqref{f:eq} is equivalent to the range equation
\begin{align}
L_{\omega,\alpha,\gamma}w=\Pi_{W}F(\omega,v+w)\label{E:range}
\end{align}
and the bifurcation equation
\begin{align}
L_{\omega,\alpha,\gamma}v=\Pi_{V}F(\omega,v+w). \label{E:bifurcation}
\end{align}

In the space $V$, one has that for $\phi=(\phi_1,\phi_2)\in\mathbb{R}^2$,
\begin{align*}
v(\theta)=&2\Re(v_{j^*_1,0})\cos( j^*_1\theta_1)+2\Im(v_{j^*_1,0})\sin( j^*_1\theta_1)+2\Re(v_{0,j^*_2})\cos( j^*_2\theta_2)+2\Im(v_{0,j^*_2})\sin( j^*_2\theta_2)\\
=&2\sqrt{(\Re(v_{j^*_1,0}))^2+(\Im(v_{j^*_1,0}))^2}\cos( j^*_1\theta_1+\phi_1)+2\sqrt{(\Re(v_{0,j^*_2}))^2+(\Im(v_{0,j^*_2}))^2}\cos( j^*_2\theta_2+\phi_2).
\end{align*}
Since $\varphi(\theta)$ satisfies \eqref{f:eq}, so does $\tilde{\varphi}(\theta):=\varphi(\theta_1+\phi_1,\theta_2+\phi_2)$. Based on this, we can take $\phi=(\phi_1,\phi_2)=0$. As a consequence,
\begin{align*}
v(\tilde{\rho})(\theta)={\tilde{\rho}}_1\cos( j^*_1\theta_1)+{\tilde{\rho}}_2\cos( j^*_2\theta_2),
\end{align*}
where
\begin{align*}
{\tilde{\rho}}_1=2\sqrt{(\Re(v_{j^*_1,0}))^2+(\Im(v_{j^*_1,0}))^2},\quad{\tilde{\rho}}_2=2\sqrt{(\Re(v_{0,j^*_2}))^2+(\Im(v_{0,j^*_2}))^2}.
\end{align*}
This is equivalent to
\begin{align}\label{f:bir-solution}
v(\rho)(\theta)=\rho_1(e^{\mathrm{i} j^*_1\theta_1}+e^{-\mathrm{i} j^*_1\theta_1})+\rho_2(e^{\mathrm{i} j^*_2\theta_2}+e^{-\mathrm{i} j^*_2\theta_2})
\end{align}
for some scalar $\rho=(\rho_1,\rho_2)$, where
\begin{align*}
\rho_1=\sqrt{(\Re(v_{j^*_1,0}))^2+(\Im(v_{j^*_1,0}))^2},\quad\rho_2=\sqrt{(\Re(v_{0,j^*_2}))^2+(\Im(v_{0,j^*_2}))^2}.
\end{align*}
If we plug expression \eqref{f:bir-solution} back into \eqref{E:range}--\eqref{E:bifurcation}, then
\begin{align}
&L_{\omega,\alpha,\gamma}w=\Pi_{W}F(\omega,v(\rho)+w),\label{d3}\\
&L_{\omega,\alpha,\gamma}v(\rho)=\Pi_{V}F(\omega,v(\rho)+w).\label{d4}
\end{align}
For $(\rho,\omega,\alpha,\gamma)\approx(0,\omega_{j^*},0,0)$, our task now is to solve the range equation \eqref{d3} and the bifurcation equation \eqref{d4}, respectively.

\section{Solutions of the range equation}\label{sec:3}
The object of this section is to look for  solutions to the range equation \eqref{d3} in the space $W\cap H^s$. The proof is based on the implicit function theorem.

For fixed $K\geq1$ large enough, we denote
\begin{align*}
J_1:=\left\{j\in J:|j|^2\geq K\right\}.
\end{align*}
Remark that $K$ is taken in the proof of Lemma \ref{le:inverse}. It is straightforward that $J=J_1\oplus J_2$ with $J_2=J\backslash J_1$. Then we further decompose $W=Y\oplus Z $,
where
\begin{align*}
\textstyle Y:= \left\{y=\sum_{j\in J_1}y_{j}e^{\mathrm{i}j\cdot \theta}\in H^0\right\},\quad Z:= \left\{z=\sum_{j\in J_2}z_{j}e^{\mathrm{i}j\cdot \theta}\in H^0\right\}.
\end{align*}
Corresponding to the above decomposition, we split up \eqref{d3} into
\begin{align}
&L_{\omega,\alpha,\gamma}y-\Pi_{Y}F(\omega,v(\rho)+y+z)=0,\label{f:range}\\
&L_{\omega,\alpha,\gamma}z-\Pi_{Z}F(\omega,v(\rho)+y+z)=0.\label{f:range2}
\end{align}
 It is clear that
\begin{align*}
(L_{\omega,\alpha,\gamma}y)(\theta)=\sum\limits_{j\in J_1}{\Theta(j,\omega,\alpha,\gamma)}{y_{j}}e^{\mathrm{i}j\cdot\theta}, \quad \forall y\in Y\cap H^{s+5},
\end{align*}
where
\begin{align}\label{f:eta}
\Theta(j,\omega,\alpha,\gamma):=&-(\omega_1j_1+\omega_2j_2)^2+\mu(\nu^2_1j^{2}_1+\nu^2_2j^{2}_2)^2+m
+\mathrm{i}\alpha(\omega_1j_1+\omega_2j_2)\nonumber\\
&+\mathrm{i}\gamma(\omega_1j_1+\omega_2j_2)(\nu^2_1j^{2}_1+\nu^2_2j^{2}_2)^2.
\end{align}
Moreover, denote by $\mathcal{B}_\varrho(\omega_{j^*})$ a neighborhood  of $\omega_{j^*}$ in $\mathbb{R}^+\times\mathbb{R}^+$, where
\begin{align}\label{f:setB}
\mathcal{B}_\varrho(\omega_{j^*}):=\left\{\omega=(\omega_1,\omega_2)\in\mathbb{R}^+\times\mathbb{R}^+:|\omega_k-\omega_{j^*_k}|<\varrho,k=1,2\right\}.
\end{align}
We first check the invertibility of the operator $L_{\omega,\alpha,\gamma}$ restricted to $Y\cap H^{s+5}$.
\begin{lemm}\label{le:inverse}
Let $s>0$. Then for all $\omega\in\mathcal{B}_\varrho(\omega_{j^*})$ and $(\alpha,\gamma)\in\mathbb{R}^2$, the linear operator $L_{\omega,\alpha,\gamma}:Y\cap H^{s+5}\longrightarrow Y\cap H^{s}$ is invertible with
\begin{align*}
L^{-1}_{\omega,\alpha,\gamma}:Y\cap H^{s}\longrightarrow Y\cap H^{s+2}.
\end{align*}
In addition, for all $j\in J_1$, $\omega\in\mathcal{B}_\varrho(\omega_{j^*})$ and $(\alpha,\gamma)\in\mathbb{R}^2$, there exists some constant $K\geq1$ large enough such that
\begin{align}\label{e3}
|\Theta(j,\omega,\alpha,\gamma)|\geq  K.
\end{align}
\end{lemm}
\begin{proof}
Obviously, $L_{\omega,\alpha,\gamma}$ is a linear operator from $Y\cap H^{s+5}$ to $Y\cap H^{s}$. Suppose that the operator $L_{\omega,\alpha,\gamma}$ could be invertible. Then for $y\in Y\cap H^s$, its inverse operator is
\begin{align*}
L^{-1}_{\omega,\alpha,\gamma}y=\sum_{j\in{J}_1}\frac{1}{\Theta(j,\omega,\alpha,\gamma)}y_je^{\mathrm{i}j\cdot \theta}.
\end{align*}
Let us check formula \eqref{e3}. For all $\omega\in\mathcal{B}_\varrho(\omega_{j^*})$, a simple computation yields that $|\omega|\leq C$ for some positive constant $C=C(\omega_{j^*},\varrho)$. If we take $|j|^2\geq K\geq \frac{2C^2+1}{\mu(\min\{\nu_1,\nu_2\})^4}$, then for all $j\in J_1$, $\omega\in\mathcal{B}_\varrho(\omega_{j^*})$ and $(\alpha,\gamma)\in\mathbb{R}^2$,
\begin{align}
|\Theta(j,\omega,\alpha,\gamma)|\geq&|-(\omega_1j_1+\omega_2j_2)^2+\mu(\nu^2_1j^{2}_1+\nu^2_2j^{2}_2)^2+m|\nonumber\\
\geq &\mu(\nu^2_1j^{2}_1+\nu^2_2j^{2}_2)^2+m-(\omega_1j_1+\omega_2j_2)^2\nonumber\\
\geq& \mu(\min\{\nu_1,\nu_2\})^4(j^2_1+j^2_2)^2-C^2(j^2_1+j^2_2+2|j_1j_2|)\nonumber\\
\geq&(j^2_1+j^2_2)(\mu(\min\{\nu_1,\nu_2\})^4(j^2_1+j^2_2)-2C^2)\geq j^2_1+j^2_2\label{f:inverse1}\\
\geq& K.\nonumber
\end{align}
Hence formula \eqref{e3} holds. As a result, $L_{\omega,\alpha,\gamma}$ is invertible.

It remains to prove that for all $y\in Y\cap H^s$,
\begin{align*}
I:=\sum\limits_{j\in J_1}\frac{1+(j^2_1+j^2_2)^{s+2}}{|\Theta(j,\omega,\alpha,\gamma)|^2}|y_{j}|^2<\infty.
\end{align*}
It follows from the definition of $J_1$ and \eqref{f:inverse1} that
\begin{align*}
I\leq\sum\limits_{j\in J_1}\frac{1+(j^2_1+j^2_2)^{s+2}}{(j^2_1+j^2_2)^2}|y_{j}|^2\leq\sum\limits_{j\in J_1}(\frac{1}{K^2}+(j^2_1+j^2_2)^{s})|y_{j}|^2\leq\|y\|^2_s<\infty.
\end{align*}
Thus we arrive at the conclusion of the lemma.
\end{proof}
 Denote by
$\mathcal{L}(H^{s_1};H^{s_2})$ the space of continuous linear operators from $H^{s_1}$ to $H^{s_2}$. In particular, we write $\mathcal{L}(H^{s};H^{s})$ as $\mathcal{L}(H^{s})$. The following lemma addresses how the operator $L^{-1}_{\omega,\alpha,\gamma}$ varies with respect to parameters $\omega,\alpha,\gamma$.
\begin{lemm}\label{le:continuous}
Let $s>0$.  The mapping $\mathcal{B}_\varrho(\omega_{j^*})\times\mathbb{R}^2\ni(\omega,\alpha,\gamma)\longmapsto L^{-1}_{\omega,\alpha,\gamma}\in \mathcal{L}(Y\cap H^s)$ is continuous with respect to the uniform operator topology.
\end{lemm}
\begin{proof}
We shall adopt the similar procedure as in the proof of Lemma \ref{le:inverse}.
\end{proof}
Because of Lemma \ref{le:inverse}, equation \eqref{f:range} turns into
\begin{align}\label{f:range3}
y-L^{-1}_{\omega,\alpha,\gamma}\Pi_{Y}F(\omega,v(\rho)+y+z)=0.
\end{align}
Due to Lemma \ref{le:smoothness} and Lemma \ref{le:inverse}, for $s\geq3$, we define a mapping as follows
\begin{align*}
G_1:&\mathbb{R}^2\times\mathcal{B}_\varrho(\omega_{j^*})\times\mathbb{R}^2\times(Z\cap H^s)\times(Y\cap H^s)\longrightarrow Y\cap H^s,\\
&(\rho,\omega,\alpha,\gamma,z,y)\longmapsto y-L^{-1}_{\omega,\alpha,\gamma}\Pi_{Y}F(\omega,v(\rho)+y+z).
\end{align*}
Notice that the expression of $v(\rho)$ is given by \eqref{f:bir-solution}.
\begin{prop}\label{prop:range}
Let  $s\geq3$. Then equation \eqref{f:range3} admits a solution
\begin{align*}
y=y(\rho,\omega,\alpha,\gamma,z)\in Y\cap H^s
\end{align*}
in a neighborhood  of $\sigma_0$ with $\sigma_0=(0,\omega_{j^*},0,0,0)$. Moreover, $y$, $\partial_{\rho_1}y,\partial_{\rho_2}y$ and $\mathrm{D}_zy$ are continuous near $\sigma_0$ with respect to $\rho,\omega,\alpha,\gamma,z$ . In particular, one has that for $(\rho,z)=(0,0)$,
\begin{align}\label{f:yyyy}
y(0,\omega,\alpha,\gamma,0)=0,\quad \partial_{\rho_k}y(0,\omega,\alpha,\gamma,0)=0,\quad k=1,2.
\end{align}
Finally, there exists an $s$-independent neighborhood $\mathcal{B}_r(\sigma_0)$ of $\sigma_0$ with $r>0$ such that equation \eqref{f:range3} has a unique solution in $C^\infty(\mathbb{T}^2;\mathbb{R})$ which coincides with the solution in $Y\cap H^s$.
\end{prop}
\begin{proof}
It is apparent from the definition of $G_1$ that $G_1(0,\omega_{j^*},0,0,0,0)=0$. By Lemma \ref{le:smoothness} and Lemma \ref{le:continuous}, the mapping $G_1$ varies continuously in $\rho,\omega,\alpha,\gamma,z,y$. Moreover, $\partial_{\rho_k}G_1,\mathrm{D}_zG_1,\mathrm{D}_yG_1$ exist and are continuous with regard to $\rho,\omega,\alpha,\gamma,z,y$. Observe that for all $\mathrm{y}\in Y\cap H^s$,
\begin{align*}
\mathrm{D}_yF(\omega,v(\rho)+y+z)[\mathrm{y}]\stackrel{\text{Lemma \ref{le:smoothness}} }{=}\lambda(2p+1)((\omega\cdot\nabla)(v(\rho)+y+z))^{2p}(\omega\cdot\nabla)\mathrm{y}.
\end{align*}
Evidently, for $(\rho,z,y)=(0,0,0)$, one arrives at $\mathrm{D}_yF(\omega,0)[\mathrm{y}]=0$.
As a result,
\begin{align*}
\mathrm{D}_{y}G_1(0,\omega_{j^*},0,0,0,0)[\mathrm{y}]=\mathrm{y}.
\end{align*}
Hence, in view of the implicit function theorem, there is a neighborhood  of $\sigma_0$ such that $y(\rho,\omega,\alpha,\gamma,z)\in Y\cap H^s$ is a solution of  equation \eqref{f:range3}. Moreover, $y$, $\partial_{\rho_k}y$ and $\mathrm{D}_zy$ vary continuously in $\rho,\omega,\alpha,\gamma,z$. Note that we can also get the uniqueness property of solutions to equation \eqref{f:range3} coming from the implicit function theorem.

In addition, since
\begin{align*}
G_1(0,\omega,\alpha,\gamma,0,0)=-L^{-1}_{\omega,\alpha,\gamma}\Pi_{Y}F(\omega,0)=0,
\end{align*}
from  uniqueness, it follows that $y(0,\omega,\alpha,\gamma,0)=0$. By  differentiating the implicit equation
\begin{align*}
G_1(\rho,\omega,\alpha,\gamma,z,y(\rho,\omega,\alpha,\gamma,z))=0
\end{align*}
with respect to $\rho_k,k=1,2$, the term $\partial_{\rho_k}G_1(\rho,\omega,\alpha,\gamma,z,y(\rho,\omega,\alpha,\gamma,z))$ is equal to
\begin{align*}
\partial_{\rho_k}y(\rho,\omega,\alpha,\gamma,z)
-L^{-1}_{\omega,\alpha,\gamma}\Pi_{Y}\lambda(2p+1)((\omega\cdot\nabla)(v(\rho)+y(\rho,\omega,\alpha,\gamma,z)+z))^{2p}&\\
\times(\omega\cdot\nabla) (\partial_{\rho_k} v(\rho)+\partial_{\rho_k} y(\rho,\omega,\alpha,\gamma,z))&.
\end{align*}
It is straightforward that $\partial_{\rho_k}G_1=0$ for $(\rho,z,\partial_{\rho_k} y)=(0,0,0)$. By virtue of uniqueness, we arrive at $\partial_{\rho_k}y(0,\omega,\alpha,\gamma,0)=0$.

In particular, we fix $\tilde{s}\geq3$, By the above discussion,  there exists a solution $\tilde{y}:\mathcal{B}_{r}(\sigma_0)\longrightarrow Y\cap H^{\tilde{s}}$ of equation \eqref{f:range3} with
\begin{align*}
\tilde{y}(\rho,\omega,\alpha,\gamma,z)=L^{-1}_{\omega,\alpha,\gamma}
\Pi_{Y}F(\omega,v(\rho)+\tilde{y}(\rho,\omega,\alpha,\gamma,z)+z).
\end{align*}
It follows from Lemma \ref{le:smoothness} and Lemma \ref{le:inverse} that
\begin{align*}
L^{-1}_{\omega,\alpha,\gamma}\Pi_{Y}F(\omega,v(\rho)+\tilde{y}(\rho,\omega,\alpha,\gamma,z)+z)\in Y\cap H^{\tilde{s}+1}.
\end{align*}
This leads to $\tilde{y}(\rho,\omega,\alpha,\gamma,z)\in  Y\cap H^{\tilde{s}+1}$. Using a direct bootstrap argument yields that
\begin{align*}
\tilde{y}(\rho,\omega,\alpha,\gamma,z)\in Y\cap H^{\tilde{s}+k},\quad\forall k\geq0.
\end{align*}
According to Sobolev embedding, we conclude $\tilde{y}(\rho,\omega,\alpha,\gamma,z)\in C^{\infty}(\mathbb{T}^2;\mathbb{R})$. As a consequence, we write $y=\tilde{y}$  by uniqueness.

Hence this ends the proof of the proposition.
\end{proof}

Because of  Proposition \ref{prop:range}, there exists a solution $y=y(\rho,\omega,\alpha,\gamma,z)\in C^{\infty}(\mathbb{T}^2;\mathbb{R})\cap (Y\cap H^s)$  with $s>0$ for equation \eqref{f:range} in an $s$-independent neighborhood of $\sigma_0$. Substituting this into equation \eqref{f:range2} gives that
\begin{align}\label{f:range4}
L_{\omega,\alpha,\gamma}z-\Pi_{Z}F(\omega,v(\rho)+y(\rho,\omega,\alpha,\gamma,z)+z)=0.
\end{align}
Our next purpose is to solve equation \eqref{f:range4}.

\begin{prop}\label{prop:range1}
Let $s>0$. Then equation \eqref{f:range4} has a solution
\begin{align*}
z=z(\rho,\omega,\alpha,\gamma)\in C^{\infty}(\mathbb{T}^2;\mathbb{R})\cap(Z\cap H^s)
\end{align*}
in an $s$-independent neighborhood $\mathcal{B}_{r_1}(\sigma_1)$ of $\sigma_1$ with $\sigma_1=(0,\omega_{j^*},0,0)$ and $r_{1}\leq r$. Furthermore, $z$, $\partial_{\rho_1}z$, and $\partial_{\rho_2}z$ are continuous near $\sigma_1$ with respect to $\rho,\omega,\alpha,\gamma$. In particular, it follows that for $\rho=0$,
\begin{align}\label{f:zzzz}
z(0,\omega,\alpha,\gamma)=0,\quad \partial_{\rho_k}z(0,\omega,\alpha,\gamma)=0,\quad k=1,2.
\end{align}
\end{prop}
\begin{proof}
Let us define a mapping as follows
\begin{align*}
G_2:\mathbb{R}^2\times\mathcal{B}_\varrho(\omega_{j^*})\times\mathbb{R}^2\times(Z\cap H^s)\longrightarrow& Z\cap H^s,\\
(\rho,\omega,\alpha,\gamma,z)\longmapsto& L_{\omega,\alpha,\gamma}z-\Pi_{Z}F(\omega,v(\rho)+y(\rho,\omega,\alpha,\gamma,z)+z).
\end{align*}
Moreover, observe that  the space $Z$ is finite dimensional by the definition of the set $J_2$.

Clearly, one has $G_2(0,\omega_{j^*},0,0,0)=0$. By Lemma \ref{le:smoothness}, the mapping $G_2$ is continuous with respect to $\rho,\omega,\alpha,\gamma,z$, and  $\partial_{\rho_k}G_2, \mathrm{D}_zG_2$ exist  varying continuously in $\rho,\omega,\alpha,\gamma,z$.  Since $Z$ is a subspace of the orthogonal complement of the kernel of the operator $L_{\omega_{j^*},0,0}$, the operator
\begin{align*}
\mathrm{D}_zG_2(0,\omega_{j^*},0,0,0)=L_{\omega_{j^*},0,0}
\end{align*}
is invertible from $Z\cap H^s$ to $Z\cap H^s$. Consequently,  according to the implicit function theorem, there exists  a neighborhood  of $\sigma_1$ such that  $z=z(\rho,\omega,\alpha,\gamma)$, with values in $Z\cap H^s$,  solves equation \eqref{f:range4}.

In addition, proceeding as in the proof of Proposition \ref{prop:range} yields that these equalities in \eqref{f:zzzz} hold and that $z(\rho,\omega,\alpha,\gamma)$ belongs to $ C^{\infty}(\mathbb{T}^2;\mathbb{R})$. We have thus proved the proposition.
\end{proof}

In conclusion, if we denote  $s>0$, then it follows from Propositions \ref{prop:range}--\ref{prop:range1} that the range equation \eqref{d3} has a solution $w=w(\rho,\omega,\alpha,\gamma)\in C^{\infty}(\mathbb{T}^2;\mathbb{R})\cap (W\cap H^s)$ in an $s$-independent neighborhood of $(0,\omega_{j^*},0,0)$. Moreover,
\begin{align*}
&w(\rho,\omega,\alpha,\gamma)=z(\rho,\omega,\alpha,\gamma)+y(\rho,\omega,\alpha,\gamma,z(\rho,\omega,\alpha,\gamma)),
\end{align*}
where $y\in C^{\infty}(\mathbb{T}^2;\mathbb{R})\cap (Y\cap H^s)$ and $z\in C^{\infty}(\mathbb{T}^2;\mathbb{R})\cap (Z\cap H^s)$ are solutions of equations \eqref{f:range}--\eqref{f:range2}, respectively.

The following proposition summarizes the existence of quasi-periodic travelling wave solutions to the range equation \eqref{d3}. Moreover, if one of the amplitudes is set to zero, then there exists a  family of rotating wave solutions with one parameter  to equation \eqref{d3}.

\begin{prop}\label{prop:travelling}
Let $s>0$. Then the range equation \eqref{d3} admits a solution  $w=w(\rho,\omega,\alpha,\gamma)$, with values in $C^{\infty}(\mathbb{T}^2;\mathbb{R})\cap (W\cap H^s)$, satisfying
\begin{flalign*}
&\mathrm{(i)}\quad w(0,\rho_2,\omega,\alpha,\gamma)\text{ is $\theta_1$-independent, i.e., } \partial_{\theta_1}w(0,\rho_2,\omega,\alpha,\gamma)(\theta)=0,&\\
&\mathrm{(ii)}\quad w(\rho_1,0,\omega,\alpha,\gamma)\text{ is $\theta_2$-independent, i.e., } \partial_{\theta_2}w(\rho_1,0,\omega,\alpha,\gamma)(\theta)=0.
\end{flalign*}
\end{prop}
\begin{proof}
We  only consider the case $\mathrm{(i)}$. The remainder of the arguments can be stated by the analogous procedure as in the proof of the case $\mathrm{(i)}$.

For $w\in W\cap H^s$, it is straightforward that
\begin{align*}
w(\theta)=w_0(\theta_2)+\tilde{w}(\theta),\quad \theta=(\theta_1,\theta_2),
\end{align*}
where $\tilde{w}(\theta)=\sum_{j\in\mathbb{Z}^2,j_1\neq0}w_{j}e^{\mathrm{i}j\cdot\theta}$. Let $\mathfrak{Y}$ be the space made of functions depending only on $\theta_2$ and $\mathfrak{Z}$ be $H^0$-orthogonal complement of $\mathfrak{Y}$. Denote by $\Pi_{\mathfrak{Y}}$ and $\Pi_\mathfrak{Z}$ the projectors onto $\mathfrak{Y}$ and $\mathfrak{Z}$, respectively. With respect to the following decomposition
\begin{align*}
W\cap H^s=(\mathfrak{Y}\cap H^s)\oplus (\mathfrak{Z}\cap H^s),
\end{align*}
by performing the Lyapunov--Schmidt reduction,  equation \eqref{d3} is equivalent to
\begin{align}
L_{\omega,\alpha,\gamma}\mathfrak{y}=\Pi_{\mathfrak{Y}}F(\omega,v(\rho)+\mathfrak{y}+\mathfrak{z}),\label{f:eq1}\\
L_{\omega,\alpha,\gamma}\mathfrak{z}=\Pi_{\mathfrak{Z}}F(\omega,v(\rho)+\mathfrak{y}+\mathfrak{z}),\label{f:eq2}
\end{align}
where $w=\mathfrak{y}+\mathfrak{z}$ with $\mathfrak{y}\in \mathfrak{Y}$, $\mathfrak{z}\in \mathfrak{Z}$.  It follows from \eqref{f:bir-solution} that for $\rho_1=0$,
\begin{align*}
v(0,\rho_2)(\theta)=2\rho_2\cos(j^*_2\theta_2).
\end{align*}
If we assume that $\mathfrak{z}|_{\rho_1=0}=0$,  then equation \eqref{f:eq2} with $\rho_1=0$ turns into
\begin{align*}
0=&\Pi_{\mathfrak{Z}}F(\omega,2\rho_2\cos(j^*_2\theta_2)+\mathfrak{y}(\theta_2)|_{\rho_1=0})\\
=&\lambda\Pi_{\mathfrak{Z}}(-2 \omega_2j^*_2\rho_2\sin( j^*_2\theta_2)+\omega_2\partial_{\theta_2}\mathfrak{y}(\theta_2)|_{\rho_1=0})^{2p+1}\\
=&0.
\end{align*}
Because of the uniqueness coming from the implicit function theorem seen in the proof of Propositions \ref{prop:range}--\ref{prop:range1},  we read that $\mathfrak{z}|_{\rho_1=0}=0$ can solve equation \eqref{f:eq2} with $\rho_1=0$. Hence equation \eqref{f:eq1} with $\rho_1=0$ is equal to
\begin{align}\label{zero}
L_{\omega,\alpha,\gamma}\mathfrak{y}(\theta_2)|_{\rho_1=0}=\Pi_{\mathfrak{Y}}F(\omega,2\rho_2\cos( j^*_2\theta_2)+\mathfrak{y}(\theta_2)|_{\rho_1=0}).
\end{align}
Observe that the subspace $\mathfrak{Y}\cap H^s$ of $W\cap H^s$ is invariant for $L_{\omega,\alpha,\gamma}$, and $F(\omega,\cdot)$. We shall adopt the similar procedure as in the proof of Propositions \ref{prop:range}--\ref{prop:range1} with $\rho_1=0$ to solve equation \eqref{zero}. By virtue of uniqueness, we obtain $w(0,\rho_2,\omega,\alpha,\gamma)=\mathfrak{y}(0,\rho_2,\omega,\alpha,\gamma)$, with $\Pi_{\mathfrak{Z}}w(0,\rho_2,\omega,\alpha,\gamma)=0$.

Thus we complete the proof of the lemma.
\end{proof}

In addition, we wish to get the smoothness of solutions of the range equation \eqref{d3} with respect to $\rho,\omega,\alpha,\gamma$.  Let us define
\begin{align*}
\Gamma(j,\omega,\alpha,\gamma):=\frac{1}{\Theta(j,\omega,\alpha,\gamma)}=:\Upsilon(\Theta(j,\omega,\alpha,\gamma)),
\end{align*}
where $\Theta$ is given by \eqref{f:eta}.

The following lemma addresses the smoothness of $L^{-1}_{\omega,\alpha,\gamma}$ with respect to $\omega,\alpha,\gamma$.

\begin{lemm}\label{le:smoothness3}
Let $\mathcal{B}_\varrho(\omega_{j^*})$ be as seen in \eqref{f:setB}. For $y\in C^{\infty}(\mathbb{T}^2;\mathbb{R})\cap (Y\cap H^s)$ with $s>0$,  the mapping $\mathcal{B}_\varrho(\omega_{j^*})\times\mathbb{R}^2\ni(\omega,\alpha,\gamma)\longmapsto L^{-1}_{\omega,\alpha,\gamma}y\in Y\cap H^s$ is $C^{\infty}$ with
\begin{align}\label{f:smoothness}
\mathrm{D}^{\delta} L^{-1}_{\omega,\alpha,\gamma}y(\theta)=\sum_{j\in J_1}\mathrm{D}^{\delta}\Gamma(j,\omega,\alpha,\gamma)y_je^{\mathrm{i}j\cdot \theta},
\end{align}
where $\mathrm{D}^{\delta}=\partial_{\omega_1}^{\delta_1}\partial_{\omega_2}^{\delta_2}
\partial_{\alpha}^{\delta_3}\partial_{\gamma}^{\delta_4}$
with $\delta_i\in\mathbb{N},i=1,\cdots,4$.
\end{lemm}
\begin{proof}
By an inductive argument,  one has
\begin{align}\label{f:zeta}
\textstyle \mathrm{D}^{\delta}\Gamma=\sum^{|\delta|}_{k=1}\Upsilon^{(k)}(\Theta)P_{k}[\mathrm{D}^{\chi_1}\Theta,\cdots,\mathrm{D}^{\chi_{n(\delta)}}\Theta], \quad|\delta|=\sum^4_{i=1}\delta_i.
\end{align}
Observe that $P_k$ is a polynomial in $n(\delta)$ variables of order at most $|\delta|$, that is
\begin{align*}
\textstyle P_{k}[\mathrm{D}^{\chi_1}\Theta,\cdots,\mathrm{D}^{\chi_{n(\delta)}}\Theta]=\sum_{|\zeta|\leq|\delta|}C^{k}_{\zeta}
(\mathrm{D}^{\chi_1}\Theta)^{\zeta_1}(\mathrm{D}^{\chi_{2}}\Theta)^{\zeta_2}\cdots (\mathrm{D}^{\chi_{n(\delta)}}\Theta)^{\zeta_{n(\delta)}},
\end{align*}
where $n(\delta)$ is the number of partial derivatives of $\Theta$ with respect to $\omega,\alpha,\gamma$ of order at most $|\delta|$,  $\chi_{i},i=1,\cdots,n(\delta)$ are multi-indices of order at most $|\delta|$, and $\zeta=(\zeta_1,\cdots,\zeta_{n(\delta)})$ is an $n(\delta)$-tuple of nonnegative integers.

Let $\Omega$ be any bounded open set in $\mathbb{R}^2$. If we could show that there exists some positive constant $C=C(\varrho,\Omega,\delta,\nu_1,\nu_2)$ such that for all $j\in J_1$ and $(\omega,\alpha,\gamma)\in\mathcal{B}_\varrho(\omega_{j^*})\times \Omega$,
\begin{align}\label{f:d_zeta}
|\mathrm{D}^{\delta}\Gamma(j,\omega,\alpha,\gamma)|\leq C|j|^{6|\delta|},\quad\forall |\delta|\geq1,
\end{align}
then \eqref{f:smoothness} will be proved by induction. We further derive the continuity of the corresponding partial derivatives.

Suppose that \eqref{f:smoothness} holds for $\mathrm{p}\geq1$ (note that the case $\mathrm{p}=1$ may be handled in the same way). Denote by $\tau=(\tau_1,\cdots,\tau_4)$ a multi-index with $|\tau|=\mathrm{p}+1$. Without loss of generality, we let  $\tau_3\geq1$. Moreover, denoting  $\varsigma=(\tau_1,\tau_2,\tau_3-1,\tau_4)$, it is straightforward that $|\varsigma|=\mathrm{p}$. For $y\in C^{\infty}(\mathbb{T}^2;\mathbb{R})\cap (Y\cap H^s)$, one carries out
\begin{align*}
\frac{1}{\upsilon^2}\|\mathrm{D}^{\varsigma}L^{-1}_{\omega,\alpha+\upsilon,\gamma}y-\mathrm{D}^{\varsigma}L^{-1}_{\omega,\alpha,\gamma}y
-\upsilon\partial_{\alpha}\mathrm{D}^{\varsigma}L^{-1}_{\omega,\alpha,\gamma}y\|^2_s
=\frac{1}{\upsilon^2}\sum_{j\in J_1}(1+|j|^{2s})|y_{j}|^2|R(j,\omega,\alpha,\gamma,\upsilon)|^2&,
\end{align*}
where
\begin{align*}
R(j,\omega,\alpha,\gamma,\upsilon)=\mathrm{D}^{\varsigma}\Gamma(j,\omega,\alpha+\upsilon,\gamma)
-\mathrm{D}^{\varsigma}\Gamma(j,\omega,\alpha,\gamma)
-\upsilon\partial_{\alpha}\mathrm{D}^{\varsigma}\Gamma(j,\omega,\alpha,\gamma).
\end{align*}
Moreover,
\begin{align*}
|R(j,\omega,\alpha,\gamma,\upsilon)|=&\textstyle| \int^1_0\partial_{\alpha}\mathrm{D}^{\varsigma}\Gamma(j,\omega,\alpha+\mathfrak{v}\upsilon,\gamma)-\partial_{\alpha}\mathrm{D}^{\varsigma}\Gamma(j,\omega,\alpha,\gamma)\mathrm{d}\mathfrak{v}||\upsilon|\\
\leq&\max_{\mathfrak{v}\in[0,1]}|\partial_{\alpha}\mathrm{D}^{\varsigma}\Gamma(j,\omega,\alpha+\mathfrak{v}\upsilon,\gamma)-\partial_{\alpha}\mathrm{D}^{\varsigma}\Gamma(j,\omega,\alpha,\gamma)||\upsilon|\\
=&|\partial_{\alpha}\mathrm{D}^{\varsigma}\Gamma(j,\omega,\tilde{\alpha},\gamma)-\partial_{\alpha}\mathrm{D}^{\varsigma}\Gamma(j,\omega,\alpha,\gamma)||\upsilon|
\end{align*}
with $\tilde{\alpha}\in[\alpha,\alpha+\upsilon]$. If $\upsilon$ is taken small enough, then $ (\tilde{\alpha},\gamma)\in \Omega$. Hence it can be seen from \eqref{f:d_zeta}  that
\begin{align*}
|\partial_{\alpha}\mathrm{D}^{\varsigma}\Gamma(j,\omega,\tilde{\alpha},\gamma)-\partial_{\alpha}\mathrm{D}^{\varsigma}\Gamma(j,\omega,\alpha,\gamma)|\leq 2C| j|^{6(\mathrm{p}+1)}.
\end{align*}
Because of the fact $C^{\infty}(\mathbb{T}^2;\mathbb{R})=\cap_{s\geq0} H^s$, we conclude that for $\upsilon$  small enough,
\begin{align*}
\textstyle\frac{1}{\upsilon^2}\|\mathrm{D}^{\varsigma}L^{-1}_{\omega,\alpha+\upsilon,\gamma}y-\mathrm{D}^{\varsigma}L^{-1}_{\omega,\alpha,\gamma}y
-\upsilon\partial_{\omega}\mathrm{D}^{\varsigma}L^{-1}_{\omega,\alpha,\gamma}y\|^2_s
\leq4C^2\sum_{j\in J_1}(1+|j|^{2s})|y_{j}|^2| j|^{12(\mathrm{p}+1)}<\infty.
\end{align*}
This gives that
\begin{align*}
\mathrm{D}^{\tau}L^{-1}_{\omega,\alpha,\gamma}y=\partial_{\alpha}\mathrm{D}^{\varsigma}L^{-1}_{\omega,\alpha,\gamma}y.
\end{align*}
By using the similar procedure as above, we can obtain the continuity of the partial derivatives with respect to $\omega,\alpha,\gamma$.

Finally, let us prove formula \eqref{f:d_zeta}. The definition of $\Theta$ shows that $\mathrm{D}^{\tau}\Theta(j,\omega,\alpha,\gamma)=0$ for all $|\tau|\geq3$. It is obvious that
\begin{align*}
|\mathrm{D}^{\tau}\Theta(j,\omega,\alpha,\gamma)|\leq C_1(1+(\max\{\nu_1,\nu_2\})^4)|j|^{6},\quad\forall |\tau|\leq2,j\in J_1.
\end{align*}
As a result,
\begin{align*}
|P_{k}[\mathrm{D}^{\chi_1}\Theta,\cdots,\mathrm{D}^{\chi_{n(\delta)}}\Theta]|\leq&\textstyle
\sum_{|\zeta|\leq|\delta|}|C^{k}_{\zeta}|
|(\mathrm{D}^{\chi_1}\Theta)^{\zeta_1}||(\mathrm{D}^{\chi_{2}}\Theta)^{\zeta_2}|\cdots |(\mathrm{D}^{\chi_{n(\delta)}}\Theta)^{\zeta_{n(\delta)}}|\\
\leq&\textstyle\sum_{|\zeta|\leq|\delta|}|C^{k}_{\zeta}|C_1^{|\zeta|}(1+(\max\{\nu_1,\nu_2\})^4)^{|\zeta|}| j |^{6|\zeta|}
\\
\leq &C'_k(1+(\max\{\nu_1,\nu_2\})^4)^{|\delta|}| j|^{6|\delta|}.
\end{align*}
According to the fact $\Upsilon(\Theta)=\frac{1}{\Theta}$, it follows that $|\Upsilon^{(k)}(\Theta)|\leq\frac{C''_k}{|\Theta|^{k+1}}$. Therefore using \eqref{e3} and \eqref{f:zeta} yields that
\begin{align*}
|\mathrm{D}^{\delta}\Gamma(j,\omega,\alpha,\gamma)|{\leq}&\textstyle\sum^{|\delta|}_{k=1}\frac{C'_k(1+(\max\{\nu_1,\nu_2\})^4)^{|\delta|} C''_k}{K^{k+1}}| j|^{6|\delta|}
\leq C| j|^{6|\delta|}.
\end{align*}
Thus we get the conclusion of the lemma.
\end{proof}
\begin{prop}\label{prop:range2}
Let  $s>0$. For $w\in C^{\infty}(\mathbb{T}^2;\mathbb{R})\cap (W\cap H^s)$, then  there is an $s$-independent neighborhood $\mathcal{B}_{r_1}(\sigma_1)$ of $\sigma_1$ with $\sigma_1=(0,\omega_{j^*},0,0)$ such that the mapping $\mathbb{R}^2\times\mathcal{B}_\varrho(\omega_{j^*})\times\mathbb{R}^2\ni(\rho,\omega,\alpha,\gamma)\longmapsto w(\rho,\omega,\alpha,\gamma)\in W\cap H^s$  is $C^\infty$.
\end{prop}
\begin{proof}
By Propositions \ref{prop:range}--\ref{prop:range1}, we obtain $w=y+z$,  where $y$ and $z$ are solutions of \eqref{f:range}--\eqref{f:range2}, respectively. Then it follows from Lemma \ref{le:smoothness}, Lemma \ref{le:smoothness3}, and the definition of $G_1$ that for all $y\in C^{\infty}(\mathbb{T}^2;\mathbb{R})\cap(Y\cap H^s)$ and $z\in C^{\infty}(\mathbb{T}^2;\mathbb{R})\cap(Z\cap H^s)$, the mapping
\begin{align*}
(\rho,\omega,\alpha,\gamma,z,y)\longmapsto G_1(\rho,\omega,\alpha,\gamma,z,y)
\end{align*}
is $C^{\infty}$  with respect to $\rho,\omega,\alpha,\gamma,z$. As a consequence, the implicit function theorem implies the $C^{\infty}$ smoothness of the mapping $(\rho,\omega,\alpha,\gamma,z)\longmapsto y(\rho,\omega,\alpha,\gamma,z)$. Moreover, since the space $Z$ is finite dimensional,  we can show that for all $z\in C^{\infty}(\mathbb{T}^2)\cap(Z\cap H^s)$, the mapping
\begin{align*}
(\rho,\omega,\alpha,\gamma,z)\longmapsto G_2(\rho,\omega,\alpha,\gamma,z)
\end{align*}
varies in a  $C^{\infty}$ way with respect to $\rho,\omega,\alpha,\gamma$. Hence it follows from the  implicit function theorem  that the mapping $(\rho,\omega,\alpha,\gamma)\longmapsto z(\rho,\omega,\alpha,\gamma)$ is $C^\infty$. This ends the proof of the proposition.
\end{proof}

\section{Solutions of the bifurcation equation}\label{sec:4}

The present section is devoted to solving the bifurcation equation \eqref{d4}. In Section \ref{sec:3}, we have sought the traveling wave solutions $w=w(\rho,\omega,\alpha,\gamma)$ to the range equation \eqref{d3}. Now we have to plug both solutions $w=w(\rho,\omega,\alpha,\gamma)$ and formula \eqref{f:bir-solution} back into equation \eqref{d4}.
Then equation \eqref{d4} is equivalent to the following system
\begin{align}\label{f:system}
\left\{ \begin{aligned}
&\rho_1L_{\omega,\alpha,\gamma}e^{\mathrm{i}j^*_1\theta_1}=\frac{1}{4\pi^2}\int_{\mathbb{T}^2}F(\omega,v(\rho)+w(\rho,\omega,\alpha,\gamma))e^{-\mathrm{i}j^*_1\theta_1}\mathrm{d}\theta e^{\mathrm{i}j^*_1\theta_1},\\
&\rho_1L_{\omega,\alpha,\gamma}e^{-\mathrm{i}j^*_1\theta_1}=\frac{1}{4\pi^2}\int_{\mathbb{T}^2}F(\omega,v(\rho)+w(\rho,\omega,\alpha,\gamma))e^{\mathrm{i}j^*_1\theta_1}\mathrm{d}\theta e^{-\mathrm{i}j^*_1\theta_1},\\
&\rho_2L_{\omega,\alpha,\gamma}e^{\mathrm{i}j^*_2\theta_2}=\frac{1}{4\pi^2}\int_{\mathbb{T}^2}F(\omega,v(\rho)+w(\rho,\omega,\alpha,\gamma))e^{-\mathrm{i}j^*_2\theta_2}\mathrm{d}\theta e^{\mathrm{i}j^*_2\theta_2},\\
&\rho_2L_{\omega,\alpha,\gamma}e^{-\mathrm{i}j^*_2\theta_2}=\frac{1}{4\pi^2}\int_{\mathbb{T}^2}F(\omega,v(\rho)+w(\rho,\omega,\alpha,\gamma))e^{\mathrm{i}j^*_2\theta_2}\mathrm{d}\theta e^{-\mathrm{i}j^*_2\theta_2}.
\end{aligned}\right.
\end{align}
Moreover, we define
\begin{align*}
\mathcal{G}^{+}_k(\rho,\omega,\alpha,\gamma):=\frac{1}{4\pi^2}\int_{\mathbb{T}^2}F(\omega,v(\rho)+w(\rho,\omega,\alpha,\gamma))\cos(j^*_k\theta_k)\mathrm{d}\theta,\quad k=1,2,\\
\mathcal{G}^{-}_k(\rho,\omega,\alpha,\gamma):=\frac{1}{4\pi^2}\int_{\mathbb{T}^2}F(\omega,v(\rho)+w(\rho,\omega,\alpha,\gamma))\sin(j^*_k\theta_k)\mathrm{d}\theta,\quad k=1,2,
\end{align*}
and
\begin{align*}
\tilde{\mathcal{G}}^{\pm}_1(\rho,\omega,\alpha,\gamma):=
\left\{ \begin{aligned}
&{\mathcal{G}}^{\pm}_1(\rho_1,\rho_2,\omega,\alpha,\gamma)/{\rho_1},\quad\rho_1\neq0,\\
&\partial_{\rho_1}{\mathcal{G}}^{\pm}_1(0,\rho_2,\omega,\alpha,\gamma),\quad\rho_1=0,
\end{aligned}\right.
\\
\tilde{\mathcal{G}}^{\pm}_2(\rho,\omega,\alpha,\gamma):=
\left\{ \begin{aligned}
&{\mathcal{G}}^{\pm}_2(\rho_1,\rho_2,\omega,\alpha,\gamma)/{\rho_2},\quad\rho_2\neq0,\\
&\partial_{\rho_2}{\mathcal{G}}^{\pm}_2(\rho_1,0,\omega,\alpha,\gamma),\quad\rho_2=0.
\end{aligned}\right.
\end{align*}
Now let us check the smoothness of the above functions $\mathcal{G}^{\pm}_k,\tilde{\mathcal{G}}^{\pm}_k,k=1,2$.
\begin{lemm}\label{le:smoothness2}
The functions $\mathcal{G}^{\pm}_k,k=1,2$ are $C^{\infty}$ with respect to $\rho,\omega,\alpha,\gamma$ in an $s$-independent neighborhood of $(0,\omega_{j^*},0,0)$ satisfying
\begin{align*}
\mathcal{G}^{\pm}_1(0,\rho_2,\omega,\alpha,\gamma)=0,\quad\mathcal{G}^{\pm}_2(\rho_1,0,\omega,\alpha,\gamma)=0.
\end{align*}
In addition, the functions $\tilde{\mathcal{G}}^{\pm}_k,k=1,2$ are also $C^{\infty}$ with respect to $\rho,\omega,\alpha,\gamma$ with, $k=1,2,i=1,2$,
\begin{align*}
&\tilde{\mathcal{G}}^{\pm}_k(0,\omega,\alpha,\gamma)=0,\quad
\partial_{\alpha}\tilde{\mathcal{G}}^{\pm}_k(0,\omega,\alpha,\gamma)=0,\quad \partial_{\gamma}\tilde{\mathcal{G}}^{\pm}_k(0,\omega,\alpha,\gamma)=0,
\quad\partial_{\omega_i}\tilde{\mathcal{G}}^{\pm}_k(0,\omega,\alpha,\gamma)=0.
\end{align*}
\end{lemm}
\begin{proof}
Combining Lemma \ref{le:smoothness} with Proposition \ref{prop:range2} yields that $\mathcal{G}^{\pm}_k,k=1,2$ are $C^{\infty}$ in $\rho,\omega,\alpha,\gamma$. Moreover, it follows from Proposition \ref{prop:travelling} that
\begin{align*}
\mathcal{G}^{+}_1(0,\rho_2,\omega,\alpha,\gamma)=\frac{1}{4\pi^2}\int_{\mathbb{T}^2}&F(\omega,2\rho_2\cos( j^*_2\theta_2)+w(0,\rho_2,\omega,\alpha,\gamma)(\theta_2))\cos(j^*_1\theta_1)\mathrm{d}\theta=0.
\end{align*}
Similarly,
\begin{align*}
\mathcal{G}^{-}_1(0,\rho_2,\omega,\alpha,\gamma)=0,\quad\mathcal{G}^{\pm}_2(\rho_1,0,\omega,\alpha,\gamma)=0.
\end{align*}

It remains to investigate  the smoothness of $\tilde{\mathcal{G}}^{\pm}_k,k=1,2$. For the sake of brevity, we just verify the $C^\infty $ smoothness of $\tilde{\mathcal{G}}^{+}_1$  with respect to $\rho,\omega,\alpha,\gamma$, and that $\tilde{\mathcal{G}}^{+}_1$ together with partial derivatives of $\tilde{\mathcal{G}}^{+}_1$ with respect to $\alpha,\gamma,\omega_i,i=1,2$ vanishes at $\rho=0$. These properties on $\tilde{\mathcal{G}}^{-}_1,\tilde{\mathcal{G}}^{\pm}_2$ can be proved in a similar way.
By the Taylor expansion of $\mathcal{G}^{+}_1$ at $\rho_1=0$, we obtain
\begin{align*}
\mathcal{G}^{+}_1(\rho_1,\rho_2,\omega,\alpha,\gamma)=&\mathcal{G}^{+}_1(0,\rho_2,\omega,\alpha,\gamma)+\partial_{\rho_1}\mathcal{G}^{+}_1(0,\rho_2,\omega,\alpha,\gamma)\rho_1+O(\rho^2_1)\\
=&\partial_{\rho_1}\mathcal{G}^{+}_1(0,\rho_2,\omega,\alpha,\gamma)\rho_1+O(\rho^2_1).
\end{align*}
This leads to
\begin{align*}
\tilde{\mathcal{G}}^{+}_1(\rho_1,\rho_2,\omega,\alpha,\gamma)=
\partial_{\rho_1}{\mathcal{G}}^{+}_1(0,\rho_2,\omega,\alpha,\gamma)+O(\rho_1).
\end{align*}
Then $\tilde{\mathcal{G}}^{+}_1$ is $C^{\infty}$ with respect to $\rho,\omega,\alpha,\gamma$.  Furthermore, using formulae \eqref{f:yyyy} and \eqref{f:zzzz} yields that
\begin{align*}
w(0,\omega,\alpha,\gamma)=0,\quad\partial_{\rho_1}w(0,\omega,\alpha,\gamma)=0.
\end{align*}
Combining these with \eqref{f:bir-solution}, Lemma \ref{le:smoothness} yields that $\partial_{\rho_1}\mathcal{G}^{+}_1(0,\omega,\alpha,\gamma)$ is equal to
\begin{align*}
&\frac{1}{4\pi^2}\int_{\mathbb{T}^2}\mathrm{D}F(\omega,v(0)+w(0,\omega,\alpha,\gamma))
\partial_{\rho_1}(v(\rho)+w(\rho,\omega,\alpha,\gamma))|_{\rho=0}
\cos(j^*_1\theta_1)\mathrm{d}\theta\\
&=\frac{\lambda(2p+1)}{4\pi^2}\int_{\mathbb{T}^2}((\omega\cdot\nabla)(v(0)+w(0,\omega,\alpha,\gamma)))^{2p}
\partial_{\rho_1}(v(\rho)+w(\rho,\omega,\alpha,\gamma))|_{\rho=0}\cos(j^*_1\theta_1)\mathrm{d}\theta\\
&=0.
\end{align*}
Hence  $\tilde{\mathcal{G}}^{+}_1(0,\omega,\alpha,\gamma)=0$. Moreover,
\begin{align*}
\partial_{\omega_i}\tilde{\mathcal{G}}^{+}_1(0,\omega,\alpha,\gamma)=\partial_{\alpha}\tilde{\mathcal{G}}^{+}_1(0,\omega,\alpha,\gamma)
=\partial_{\gamma}\tilde{\mathcal{G}}^{+}_1(0,\omega,\alpha,\gamma)
=0.
\end{align*}
Thus we get the conclusion of the lemma.
\end{proof}

In view of Lemma \ref{le:smoothness2}, splitting up into real and imaginary parts yields that the above system \eqref{f:system} can be simplified to
\begin{align}\label{f:system2}
\left\{ \begin{aligned}
&-\omega_1^2(j^*_1)^2+\mu\nu^4_1(j^*_1)^4+m=\tilde{\mathcal{G}}^{+}_1(\rho,\omega,\alpha,\gamma),\\
&-\alpha\omega_1j^*_1-\gamma\omega_{1}\nu^4_1(j^*_1)^5=\tilde{\mathcal{G}}^{-}_1(\rho,\omega,\alpha,\gamma),\\
&-\omega_2^2(j^*_2)^2+\mu\nu^4_2(j^*_2)^4+m=\tilde{\mathcal{G}}^{+}_2(\rho,\omega,\alpha,\gamma),\\
&-\alpha\omega_2j^*_2-\gamma\omega_{2}\nu^4_2(j^*_2)^5=\tilde{\mathcal{G}}^{-}_2(\rho,\omega,\alpha,\gamma).
\end{aligned}\right.
\end{align}
Remark that system \eqref{f:system2} is made of four equations in the six unknowns $(\rho_1,\rho_2,\omega_1,\omega_2,\alpha,\gamma)$.
By virtue of linearizing system \eqref{f:system2} with respect to $\omega_1,\omega_2,\alpha,\gamma$ at $(0,\omega_{j^*},0,0)$, we obtain the following  matrix
\begin{align*}
A=
\left(
\begin{array}{cccccc}
-2\omega_{j^*_1}(j^*_1)^2 & 0&   0 &0\\
0&0 &   -\omega_{j^*_1}j^*_1& -\omega_{j^*_1}\nu^4_1(j^*_1)^5 \\
0& -2\omega_{j^*_2}(j^*_2)^2 &  0& 0\\
0&0 &   -\omega_{j^*_2}j^*_2& -\omega_{j^*_2}\nu^4_2(j^*_2)^5 \\
\end{array}
\right).
\end{align*}
Some simple manipulation yields that
\begin{align*}
\det A=-4\omega^2_{j^*_1}(j^*_1)^3\omega^2_{j^*_2}(j^*_2)^3(\nu^2_2(j^*_2)^2+\nu^2_1(j^*_1)^2)(\nu_2j^*_2+\nu_1j^*_1)(\nu_2j^*_2-\nu_1j^*_1).
\end{align*}
 According to the fact $j_1^*\neq \frac{\nu_2}{\nu_1}j^*_{2}$, one has $\det A\neq0$. By means of Lemma \ref{le:smoothness2}, it follows from the implicit function theorem that the mappings
\begin{align*}
\rho \longmapsto\omega_i(\rho),\quad i=1,2,\quad \rho\longmapsto\alpha(\rho),\quad\rho\longmapsto\gamma(\rho)
\end{align*}
are $C^\infty$, respectively.

As a consequence, we complete the proof of Theorem \ref{theo1}.


\end{document}